\documentclass[11pt]{oupau}
\usepackage{color}

\usepackage{fourier}

\usepackage[pdfstartview=FitH,
            colorlinks,
           linkcolor=reference,
            citecolor=citation,
            urlcolor=e-mail]{hyperref}

\usepackage{amsmath,amscd,bbm,leftidx,dsfont,setspace}

\RequirePackage{color}
\definecolor{todo}{rgb}{1,0,0}
\definecolor{answer}{rgb}{0,0,1}
\definecolor{new}{rgb}{1,0,1}
\definecolor{conditional}{rgb}{0,1,0}
\definecolor{e-mail}{rgb}{0,.40,.80}
\definecolor{reference}{rgb}{.20,.60,.22}
\definecolor{mrnumber}{rgb}{.80,.40,0}
\definecolor{citation}{rgb}{0,.40,.80}

\theoremstyle{remark}
\newtheorem{remark}[theorem]{Remark}

\numberwithin{equation}{section}

\theoremstyle{definition}
\newtheorem{algorithm}{}
\newtheorem{step}{Step}

\def\beq{\begin{equation}}
\def\eeq{\end{equation}}

\def\calD{{\mathcal D}}
\def\calE{{\mathcal E}}
\def\Ob{{\mathrm Ob}}
\def\Z{{\mathbb Z}}

\def\Q{{\mathbb Q}}
\def\N{{\mathbb N}}
\def\ip{{\mathfrak p}}
\def\is{{\mathfrak s}}
\def\ld{{\ell\Delta}}
\def\deltabar{{\overline{\Delta}}}
\def\dbar{{\overline{\partial}}}

\newcommand{\K}{\mathbf{K}}
\newcommand{\Gm}{\mathbf{G}_{m}}
\newcommand{\G}{\mathbf{G}}
\newcommand{\ZZ}{\mathds{Z}}
\newcommand{\Ga}{\mathbf{G}_{a}}

\DeclareMathOperator{\gr}{gr}
\DeclareMathOperator{\ord}{ord}
\DeclareMathOperator{\Hom}{Hom}
\DeclareMathOperator{\Ker}{Ker}
\DeclareMathOperator{\Rep}{\bf Rep}
\DeclareMathOperator{\TRep}{\widetilde{\bf Rep}}
\DeclareMathOperator{\GL}{\bf GL}
\DeclareMathOperator{\id}{id}
\DeclareMathOperator{\Const}{\mathcal C}
\DeclareMathOperator{\I}{\mathbb I}
\DeclareMathOperator{\U}{\mathcal U}
\DeclareMathOperator{\diff}{diff}
\DeclareMathOperator{\Mn}{\bf M}

\DeclareMathOperator{\SL}{\bf SL}

\DeclareMathOperator{\Quot}{Quot}
\DeclareMathOperator{\Char}{char}
\DeclareMathOperator{\Id}{Id}

\DeclareMathOperator{\Span}{span}

\DeclareMathOperator{\diag}{diag}

\DeclareMathOperator{\Cat}{\mathcal{C}}
\DeclareMathOperator{\Ll}{\ell\ell}
\DeclareMathOperator{\Ru}{{\bf R}_u}

\DeclareMathOperator{\soc}{\mathrm soc}

\newcommand{\Le}{\leqslant}
\newcommand{\Ge}{\geqslant}

\usepackage[totalheight=7.8in, totalwidth=5.7in,paperwidth=7.3in,paperheight=9.7in,left=1.8in, top=2in]{geometry}
\begin{document}

\title{Reductive Linear Differential Algebraic Groups and the  Galois Groups of Parameterized Linear Differential Equations}
\shorttitle{Reductive LDAGs and the Galois Groups of Parameterized Linear Differential Equations}

\author{Andrey Minchenko\affil{1}, Alexey Ovchinnikov\affil{2,3}, and Michael F. Singer\affil{4}}

\abbrevauthor{A. Minchenko {\em et al.}}

\headabbrevauthor{A. Minchenko {\em et al.}}

\address{
\affilnum{1}The Weizmann Institute of Science, Department of Mathematics, Rehovot 7610001, Israel, 
\affilnum{2} Department of Mathematics, CUNY Queens College, 65-30 Kissena Blvd,
Queens, NY 11367, USA, 
\affilnum{3}
Department of Mathematics, CUNY Graduate Center, 365 Fifth Avenue,
New York, NY 10016, USA, and
\affilnum{4} Department of Mathematics,
North Carolina State University,
Raleigh, NC 27695-8205, USA}

\correspdetails{aovchinnikov@qc.cuny.edu}

\volumeyear{2014}
\paperID{rnt344}

\begin{abstract} 
We develop the representation theory for reductive linear differential algebraic groups (LDAGs). In particular, we exhibit an explicit sharp upper bound for orders of derivatives in differential representations of reductive LDAGs, extending existing results, which were obtained for $\SL_2$ in the case of just one derivation.
As an application of the above bound, we develop an algorithm that tests whether the parameterized differential Galois group of a system of linear differential 
equations is reductive and, if it is, calculates it.
\end{abstract}
\received{April 5, 2013}
\revised{November 29, 2013}
\accepted{December 2, 2013}


\maketitle

\section{Introduction}At the most basic level, a linear differential algebraic group (LDAG)  is a group of matrices whose entries are functions satisfying a fixed set of polynomial  differential equations.  An algebraic study of these objects in the context of differential algebra was initiated by Cassidy in~\cite{Cassidy}   and further developed by Cassidy~\cite{CassidyRep,Cassunipot,CassidyClassification,CassidyLie,CassidyPlane}.    This theory of LDAGs has been extended to a theory of general differential algebraic groups by Kolchin, Buium, Pillay and others. Nonetheless, interesting applications via the parameterized Picard--Vessiot (PPV) theory to questions of integrability~\cite{GO,MiSiIsom}  and hypertranscendence~\cite{PhyllisMichael,CharlotteMichael} support a more detailed study of the linear case.

Although there are several similarities between the theory of LDAGs  and the theory of linear algebraic groups (LAGs), a major difference lies in the representation theory of reductive groups.  If $G$ is a reductive LAG defined over a field of characteristic $0$, then any representation of $G$ is completely reducible, that is, any invariant subspace has an invariant complement.  This is no longer the case for reductive LDAGs. For example, if $k$ is a differential field containing at least one element whose derivative is nonzero, the reductive LDAG $\SL_2(k)$ has a representation in $\SL_4(k)$ given by 
$$A \mapsto\begin{pmatrix}
A&A'\\
0&A
\end{pmatrix}.$$
One can show that this is not completely reducible (cf. Example~\ref{ex:sharp}).  Examples such as this show that the process of taking derivatives complicates the representation theory in a significant way. 
Initial steps to understand representations of LDAGs are given in~\cite{Cassidy, CassidyRep} and a classification of  semisimple LDAGs is given in~\cite{CassidyClassification}. A Tannakian approach to the representation theory of LDAGs was introduced in~\cite{OvchRecoverGroup,OvchTannakian} (see also~\cite{Moshe,Moshe2012}) and successfully used to further our understanding of representations of reductive LDAGs in~\cite{diffreductive,MinOvRepSL2}. This Tannakian approach gives a powerful tool in which one can understand the impact of taking derivatives on the representation theory of LDAGs. 

The main results of the paper consist of  bounds for orders of derivatives in differential representations of semisimple and reductive LDAGs (Theorems~\ref{thm:EqualFiltrations} and~\ref{thm:bound}, respectively).  Simplified, our results say that, for a semisimple LDAG, the orders of derivatives are bounded by the dimension of the representation. For a reductive LDAG containing a finitely generated group dense in the Kolchin topology (cf. Section~\ref{sec:basicdef}), they are bounded by the maximum of the bound for its semisimple part and by the order of differential equations that define the torus of the group. This result completes and substantially extends what could be proved using \cite{MinOvRepSL2}, where one is restricted just to $\SL_2$, one derivation, and to those representations that are extensions of just two irreducible representations. We expect that the main results of the present paper will be used in the future to give a complete classification of differential representations of semisimple LDAGs (as this was partially done for $\SL_2$ in \cite{MinOvRepSL2}). Although reductive and semisimple differential algebraic groups were  studied in~\cite{CassidyClassification,diffreductive},  the techniques used there were not developed enough to achieve the goals of this paper.  The main technical tools that we develop and use in our paper are filtrations of  modules of reductive LDAGs, which, as we show, coincide with socle filtrations in the semisimple case (cf. \cite{BGS,Kodera}). We expect that this technique is general and powerful enough to have applications beyond this paper.

In this paper, we also apply these results to the Galois theory of parameterized linear differential equations.  The classical differential Galois theory studies symmetry
groups of solutions of linear differential equations, or,
equivalently, the groups of automorphisms of the corresponding
extensions of differential fields. The groups that arise are LAGs over the field of constants. This theory, started in the 19th century by Picard and Vessiot, was put on a firm
modern footing by Kolchin~\cite{Kolchin1948}.
 A generalized differential Galois theory that uses
Kolchin's axiomatic approach~\cite{KolDAG} and realizes differential
algebraic groups as Galois groups was initiated in~\cite{Landesman}.

 The PPV Galois theory considered by Cassidy and Singer in~\cite{PhyllisMichael} is a special case of the Landesman generalized differential Galois theory and studies symmetry groups of the solutions of linear differential equations whose coefficients contain parameters. This is done by
constructing a differential field containing the solutions and their
derivatives with respect to the parameters, called a  PPV extension, and studying its group of
differential symmetries, called a parameterized differential Galois
group. The Galois groups that arise are LDAGs which are defined by polynomial differential
equations in the parameters. Another approach to the Galois theory of systems of linear differential equations with parameters is given in~\cite{TG}, where the authors study Galois groups for generic values of the parameters. It was shown in~\cite{DreyfusDensity,MiSiIsom}  that, a necessary and sufficient condition that an LDAG $G$ is a PPV-Galois group over the field $C(x)$ is that $G$ contains a finitely generated Kolchin-dense subgroup (under some further restrictions on $C$).

In Section~\ref{sec:main}, we show how our main result yields algorithms in the PPV theory.  For systems of differential equations without parameters in the usual Picard--Vessiot theory, there are many existing algorithms for computing differential Galois groups. A complete algorithm over the field $C(x)$, where $C$ is a computable algebraically closed field of constants, $x$ is transcendental over $C$,  and its derivative is equal to $1$, is given in~\cite{HRUW} (see also~\cite{CoSi97b} for the case when the group is reductive).  More efficient algorithms  for equations of low order appear  
in ~\cite{Kovacic1,FelixMichael1,FelixMichael2,FelixMichael3,FelixJAW,Michael}.  These latter algorithms depend on knowing a list of groups that can possibly occur and step-by-step eliminating the choices.

 For parameterized systems, the first known algorithms are given in \cite{Carlos,Dreyfus}, which apply to systems of first and second orders (see also~\cite{Carlos2} for the application of these techniques to the incomplete gamma function). An algorithm for the case in which the quotient of the parameterized Galois group by its unipotent radical is constant is given in \cite{MiOvSi}. In the present paper, without any restrictions to the order of the equations,
based on our main result (upper bounds mentioned above), we present  algorithms that
\begin{enumerate}
\item compute the quotient of the parameterized Galois group $G$ by its unipotent radical $\Ru(G)$;
\item test whether $G$ is reductive (i.e., whether $\Ru(G) = \{\id\}$)
\end{enumerate}   Note that these algorithms imply that we can determine if the PPV-Galois group is reductive and, if it is, compute it.

The paper is organized as follows. We start by recalling the basic definitions of differential algebra, differential dimension, differential algebraic groups, their representations, and unipotent and reductive differential algebraic groups in Section~\ref{sec:basicdef}. The main technical tools of the paper, properties of LDAGs containing a Kolchin-dense finitely generated subgroup  and grading filtrations of differential coordinate rings, can be found in Sections~\ref{sec:DFGG} and~\ref{sec:maintech}, respectively. The main result is in Section~\ref{sec:main:bound}. The main algorithms  are described in Section~\ref{sec:main}.
Examples that show that the main upper bound is sharp and illustrate the algorithm are in Section~\ref{sec:examples}.

\section{Basic definitions}\label{sec:basicdef}
\subsection{Differential algebra}
We begin by fixing  notation and recalling some basic facts from differential algebra (cf. \cite{Kol}). In this paper a $\Delta$-ring will be a commutative associative ring $R$ with unit $1$ and commuting derivations $\Delta=\{\partial_1,\ldots,\partial_m\}$.
We let $$\Theta := \big\{\partial_1^{i_1}\cdot\ldots\cdot\partial_m^{i_m}\:|\: i_j \Ge 0\big\}$$ and note that this free semigroup acts naturally on $R$.  For an element $\partial_1^{i_1}\cdot\ldots\cdot\partial_m^{i_m} \in \Theta$, we let $$\ord\big(\partial_1^{i_1}\cdot\ldots\cdot\partial_m^{i_m}\big) := i_1+\ldots+i_m.$$
Let $Y = \{y_1,\ldots,y_n\}$ be a set of variables and
$$
\Theta Y := \left\{\theta y_j
\:\big|\: \theta\in\Theta,\ 1\Le j\Le n\right\}.
$$
The ring of differential polynomials $R\{Y\}$ in
differential indeterminates $Y$
over $R$ is
$R[\Theta Y]$
 with
the derivations $\partial_i$ that 
extends the $\partial_i$-action on $R$ as follows:
$$
\partial_i\left(\theta y_j\right) := (\partial_i\cdot\theta)y_j,\quad 1 \Le j \Le n,\ \  1\Le i\Le m.$$
 An ideal $I$ in a $\Delta$-ring $R$ is called a differential ideal if
$
\partial_i(a) \in I$ for all  $a \in I$, $1\Le i\Le m$. 
For $F \subset R$,  $[F]$ denotes the differential ideal of $R$ generated by $F$.

Let $\K$  be a $\Delta$-field of characteristic zero.  We denote the subfield of constants of $\K$ by 
$$\K^\Delta:= \{c\in \K \ | \ \partial_i(c) = 0, \ 1 \Le i \Le m\}.$$ Let $\U$ be a differentially closed field containing $\K$, that is, a $\Delta$- extension field of $\K$  such that any system of polynomial differential equations with coefficients in $\U$ having a solution in some $\Delta$-extension of $\U$ already have a solution in $\U^n$ (see \cite[Definition~3.2]{PhyllisMichael} and the references therein). 

\begin{definition} A {\it Kolchin-closed} subset $W(\U)$ of $\U^n$ over $\K$ is the set of common zeroes
of a system of differential algebraic equations with coefficients in $\K,$ that is, for $f_1,\ldots,f_l \in \K\{Y\}$, we define
$$
W(\U) = \left\{ a \in \U^n\:|\: f_1(a)=\ldots=f_l(a) = 0\right\}.$$

\noindent If $W(\U)$ is a  Kolchin-closed subset  of $\U^n$ over $\K$, we let  $\I(W) = \{ f\in \K\{y_1,  \ldots , y_n\} \ | \ f(w) = 0 \ \forall \ w\in W(\U)\}$.
\end{definition}

One has the usual correspondence between Kolchin-closed subsets of  $\K^n$ defined over $\K$ and radical differential ideals of $\K\{y_1, \ldots , y_n\}.$   Given a Kolchin-closed subset $W$ of
$\U^n$ defined over $\K$, we let the coordinate ring $\K\{W\}$ be defined as
$$
\K\{W\} = \K\{y_1,\ldots,y_n\}\big/\I(W).
$$
A differential polynomial map $\varphi : W_1\to W_2$ between Kolchin-closed subsets of $\U^{n_1}$ and $\U^{n_2}$, respectively, defined over $\K$, is given in coordinates by differential polynomials in
$\K\{W_1\}$. Moreover, to give $\varphi : W_1 \to W_2$
is equivalent to defining a differential $\K$-homomorphism $\varphi^* : \K\{W_2\} \to \K\{W_1\}$. If $\K\{W\}$ is an integral domain, then $W$ is called {\it irreducible}. This is equivalent to $\I(W)$ being a prime differential ideal. More generally, if $$\I(W) = \mathfrak{p}_1\cap\ldots\cap \mathfrak{p}_q$$ is a minimal prime decomposition, which is unique up to permutation, \cite[VII.29]{Kap}, then the irreducible Kolchin-closed sets 
$W_1,\ldots, W_q$ corresponding to $\mathfrak{p}_1,\ldots,\mathfrak{p}_q$ are called the {\it irreducible components} of $W$. We then have $$W = W_1\cup\ldots\cup W_q.$$
If $W$ is an irreducible Kolchin-closed subset of  $\U^n$ defined over $\K$, we denote the quotient field of $\K\{W\}$ by $\K\langle W\rangle$.

In the following, we shall need the notion of a Kolchin closed set being of {\it differential type  at most zero}.  The general concept of differential type is defined in terms of the Kolchin polynomial (\cite[Section~II.12]{Kol}) but this more restricted notion has a simpler definition.

\begin{definition} Let $W$ be an irreducible Kolchin-closed subset of  $\U^n$ defined over $\K$.  We say that $W$ is of {\it differential type at most zero}  and denote this by $\tau(W) \Le 0$ if $\rm{tr.~deg}_\K \K\langle W\rangle < \infty.$ If $W$ is an arbitrary Kolchin-closed subset of  $\U^n$ defined over $\K$, we say that $W$ has differential type at most zero if this is true for each of its components.
\end{definition}

We shall use the fact that if $H\trianglelefteq G$ are LDAGs, then  $\tau(H) \Le 0$ and $\tau(G/H) \Le 0$ if and only if $\tau(G) \Le 0$ \cite[Section~IV.4]{KolDAG}.

\subsection{Linear Differential Algebraic Groups} Let $\K \subset \U$ be as above. Recall that LDAG stands for linear differential algebraic group.
\begin{definition}\cite[Chapter~II, Section~1, p.~905]{Cassidy}\label{def:LDAG} An {\it LDAG} over $\K$
is a Kolchin-closed subgroup $G$ of $\GL_n(\U)$ over $\K$,
that is, an intersection
of a Kolchin-closed subset of $\U^{n^2}$ with $\GL_n(\U)$ that is closed under
the group operations.
\end{definition}

Note that we identify $\GL_n(\U)$ with a Zariski closed
subset of $\U^{n^2+1}$ given by
$$\left\{(A,a)\:\big|\: (\det(A))\cdot a-1=0\right\}.$$
If $X$ is an invertible $n\times n$ matrix, we
can identify it with the pair $(X,1/\det(X))$. Hence, we may represent the coordinate ring of $\GL_n(\U)$ as
$
\K\{X,1/\det(X)\}$. 
As usual, let $\Gm(\U)$ and $\Ga(\U)$ denote the multiplicative and additive groups of $\U$, respectively. The coordinate ring of the LDAG $\SL_2(\U)$ is isomorphic to    $$\K\{c_{11},c_{12},c_{21},c_{22}\}/[c_{11}c_{22}-c_{12}c_{21} -1].$$
For a group $G\subset\GL_n(\U)$, we denote  the Zariski closure of $G$ in $\GL_n(\U)$ by $\overline{G}$. Then $\overline{G}$ is a LAG over $\U$. If $G\subset\GL_n(\U)$ is an LDAG defined over $\K$, then $\overline{G}$ is
defined over $\K$ as well.

The irreducible component of an LDAG $G$ containing $\id$, the identity,  is called the {\it identity component} of $G$ and denoted by $G^\circ$. An LDAG $G$ is called {\it connected} if $G = G^\circ$, which is equivalent to $G$ being an irreducible Kolchin closed set \cite[p.~906]{Cassidy}.

The coordinate ring $\K\{G\}$ of an LDAG $G$ has a structure of a {\it differential Hopf algebra}, that is, a Hopf algebra in which the coproduct, antipode, and counit are homomorphisms of differential algebras \cite[Section~3.2]{OvchRecoverGroup} and \cite[Section~2]{CassidyRep}. One can view $G$ as a representable functor defined on $\K$-algebras, represented by $\K\{G\}$. For example, if $V$ is an $n$-dimensional vector space over $\K$, $\GL(V)=\mathrm{Aut} V$ is an LDAG represented by $\K\{\GL_n\}=\K\{\GL_n(\U)\}$.

\subsubsection{Representations of LDAGs}
\begin{definition}\cite{CassidyRep},\cite[Definition~6]{OvchRecoverGroup} Let $G$ be an LDAG. A  differential polynomial
group homomorphism  $$r_V : G \to \GL(V)$$ is called a
{\it differential representation} of $G$, where $V$ is a
finite-dimensional vector space over $\K$. Such space is
simply called a {\it $G$-module}. This is equivalent to giving a {\it comodule structure} $$\rho_V : V \to V\otimes_\K \K\{G\}, $$ see \cite[Definition~7 and Theorem~1] {OvchRecoverGroup}, \cite[Section~3.2]{Waterhouse}. Moreover, if $U\subset V$ is a submodule, then ${\varrho_V|}_U=\varrho_U$.

As usual, {\it morphisms} between $G$-modules are $\K$-linear maps that are $G$-equivariant. The category of differential representations of $G$ is denoted by $\Rep G$.
\end{definition}

For an LDAG $G$, let $A:=\K\{G\}$ be its differential Hopf algebra and $$\Delta : A \to A\otimes_\K A$$ be the comultiplication inducing the 
{\it right-regular $G$-module structure} on $A$ as follows (see also \cite[Section~4.1]{OvchRecoverGroup}). For $g, x \in G(\U)$ and $f \in A$,
$$
\left(r_g(f)\right)(x)=f(x\cdot g) = \Delta(f)(x,g)= \sum_{i=1}^nf_i(x) g_i(g),
$$
where $\Delta(f) = \sum_{i= 1}^n f_i\otimes g_i$.
The $k$-vector space $A$ is an $A$-comodule via
$$
\varrho_A:=\Delta.
$$
\begin{proposition}\label{prop:BasicPropsA}\cite[Corollary~3.3, Lemma~3.5]{Waterhouse}\cite[Lemma~3]{OvchRecoverGroup}
The coalgebra $A$ is a countable union of its finite-dimensional subcoalgebras. If $V\in\Rep G$, then, as an $A$-comodule, $V$ embeds into $A^{\dim V}$.
\end{proposition}

By \cite[Proposition~7]{Cassidy}, $\rho(G)\subset\GL(V)$ is a differential algebraic subgroup.  
Given a representation $\rho$ of an LDAG $G$, one can define its prolongations 
$$P_i(\rho) : G \to \GL(P_i(V))
$$ with respect to $\partial_i$ as follows (see \cite[Section~5.2]{GGO}, \cite[Definition~4 and Theorem~1]{OvchRecoverGroup}, and \cite[p.~1199]{diffreductive}). Let 
\begin{equation}\label{eq:prolongation}
P_i(V):=\leftidx{_{\K}}{\left((\K\oplus \K\partial_i)_{\K}\otimes_{\K} V\right)}
\end{equation}
as vector spaces, where $\K\oplus \K\partial_i$ is considered as the right $\K$-module:
$
\partial_i\cdot a = \partial_i(a) + a\partial_i
$
for all $a \in \K$.
Then the action of $G$ is given by $P_i(\rho)$ as follows:
$$
P_i(\rho)(g) (1\otimes v) := 1\otimes \rho(g)(v),\quad P_i(\rho)(g)(\partial_i\otimes v) := \partial_i\otimes\rho(g)(v)
$$
for all $g \in G$ and $v \in V$. In the language of matrices, if $A_g \in \GL_n$ corresponds to the action of $g \in G$  on $V$, then the matrix
$$
\begin{pmatrix}
A_g&\partial_i A_g\\
0&A_g
\end{pmatrix}
$$
corresponds to the action of $g$ on $P_i(V)$. In what follows, the $q^{\rm th}$ iterate of $P_i$ is denoted
by $P_i^q$. Moreover, the above induces the exact sequences:
\begin{equation}\label{eq:es}
\begin{CD}
0 @>>> V @>\iota_i>> P_i(V)@>\pi_i>> V @>>>0,
\end{CD}
\end{equation}
where $\iota_i(v) = 1\otimes v$ and $\pi_i(a\otimes u + b\partial_i\otimes v) = bv$, $u,\,v \in V$, $a,\, b \in \K$.
For any integer $s$, we will refer to $$P_m^sP_{m-1}^s\cdot\ldots\cdot P_1^s(\rho): G \rightarrow GL_{N_s}$$  to be the {\em $s^{th}$ total prolongation of $\rho$} (where $N_s$ is the dimension of the underlying prolonged vector space). We denote this representation by $P^s(\rho):G \rightarrow GL_{N_s}$. The underlying vector space is denoted by $P^s(V)$.

It will be convenient to consider $A$ as a $G$-module. For this, let $\TRep G$ denote the differential tensor category 
of all $A$-comodules (not necessarily finite-dimensional), which are direct limits of finite-dimensional $A$-comodules by \cite[Section~3.3]{Waterhouse}.
Then $A\in\TRep G$ by Proposition~\ref{prop:BasicPropsA}. 

\subsubsection{Unipotent radical of  differential algebraic groups and reductive LDAGs}

\begin{definition}\cite[Theorem~2]{Cassunipot} Let $G$ be an LDAG defined over $\K$. We say that $G$ is {\it unipotent} if one of 
the following conditions holds:
\begin{enumerate}
\item $G$ is conjugate to a differential algebraic subgroup  of the group ${\mathbf U}_n$ of unipotent upper triangular matrices;
\item $G$ contains no elements of finite order $>1$;
\item $G$ has a descending normal sequence  of differential algebraic subgroups 
$$G=G_0 \supset G_1 \supset \ldots \supset G_N =\{1\}$$
with $G_i/G_{i+1}$ isomorphic to a differential algebraic subgroup of the additive group $\bold{G}_a$.\qedhere
\end{enumerate}
\end{definition}

One can  show that 
an LDAG   $G$ defined over $\K$ admits a maximal normal unipotent differential subgroup \cite[Theorem~3.10]{diffreductive}.

\begin{definition}This subgroup is  called the {\it unipotent radical} of $G$ 
and denoted by $\Ru(G)$. The unipotent radical of a LAG $H$ is also denoted by $\Ru(H)$.
\end{definition}

\begin{definition}\cite[Definition~3.12]{diffreductive}
An LDAG $G$ is called {\it reductive} if its unipotent radical
is trivial, that is, $\Ru(G) = \{\id\}$.
\end{definition}

\begin{remark}\label{rem:GbarG}
If $G$ is given as a linear differential algebraic subgroup of some $\GL_\nu$, we may consider its Zariski closure $\overline{G}$
in $\GL_\nu$, which is an algebraic group scheme defined over $\K$. Then, following the proof of \cite[Theorem~3.10]{diffreductive} $$\Ru(G) = \Ru{\left(\overline{G}\right)} \cap G.$$ 
This implies that, if $\overline G$ is reductive, then $G$ is reductive.
However,  in general the Zariski closure of $\Ru(G) $ may be strictly 
included in $\Ru(\overline{G})$ \cite[Ex.~3.17]{diffreductive}.  
\end{remark}

\subsubsection{Differentially finitely generated groups}\label{sec:DFGG}   As mentioned in the introduction, one motivation for studying LDAGs is their use in the PPV theory.  In Section~\ref{sec:main}, we will discuss PPV-extensions of certain fields whose PPV-Galois groups satisfy the following property.
In this subsection, we will assume that $\K$ is differentially closed.

 \begin{definition}\label{def:DFGG} Let $G$ be an LDAG defined over $\K$.  We say that $G$ is \emph{differentially finitely generated}, or simply a \emph{DFGG}, if $G(\K)$ contains a  finitely generated subgroup that is Kolchin dense over $\K$. 
\end{definition}

\begin{proposition}\label{prop:finite}
If $G$ is a DFGG, then its identity component $G^\circ$ is a DFGG. 
\end{proposition}
\begin{proof} The Reidemeister--Schreier Theorem implies that a subgroup of finite index in a finitely generated group is finitely generated (\cite[Corollary~2.7.1]{MKS66}). One can use this fact to construct a proof of the above. Nonetheless,  we present a self-contained proof. 

Let $F:= G/G^\circ$ and $t:= |G/G^\circ|$. We claim that every sequence of $t$ elements of $F$ has a contiguous subsequence whose product is the identity.  To see this, let $a_1,\ldots ,a_t$ be a sequence of elements of $F$.  Set 
$$
b_1:= a_1, b_2:= a_1a_2, \ldots , b_t:=a_1a_2\cdot\ldots\cdot a_t.
$$
If there are $i<j$ such that $b_i=b_j$ then 
$$
{\rm id} = b_i^{-1}b_j= a_{j+1} \cdot\ldots \cdot a_j.
$$
If  the $b_j$ are pairwise distinct, they exhaust $F$ and so one of them must be the identity.

Let $S=S^{-1}$ be a finite set generating a dense subgroup $\Gamma\subset G$. Set
$$
\Gamma_0:=\big\{s \:|\: s=s_1\cdot\ldots\cdot s_m\in G^\circ,\ s_i\in S\big\}.
$$ 
Then $\Gamma_0$ is a Kolchin dense subgroup of $G^\circ$. Applying the above observation concerning $F$, we see that $\Gamma_0$ is generated by the finite set
$$
S_0:=\big\{s \:|\: s = s_1\cdot\ldots\cdot s_m\in G^\circ,\ s_i\in S\ \text{and}\  m\Le|G/G^\circ|\big\}.\qedhere
$$\end{proof}

\begin{lemma}\label{lem2}  If $H\subset \Ga^m$ is a DFGG, then $\tau(H) \Le 0$.\end{lemma}
\begin{proof} Let $\pi_i$ be the projection of $\Ga^m$ onto its $i$th factor. We have that $\pi_i(H)\subset \Ga$ is a DFGG and so, by \cite[Lemma~2.10]{MiOvSi}, $\tau(\pi_i(H)) \Le 0$. Since $$H \subset \pi_1(H)\times \ldots \times \pi_m(H)\quad \text{and}\quad \tau(\pi_1(H)\times \ldots \times \pi_m(H)) \Le 0,$$ we have $\tau(H) = 0$.\end{proof}

\begin{lemma}\label{lem3}  If $H\subset \Gm^r$ is a DFGG, then $\tau(H) \Le 0$.\end{lemma}
\begin{proof} Let $\ld:\Gm^r \rightarrow \Ga^{rm}$ be the homomorphism
$$\ld(y_1, \ldots , y_r) = \left(\frac{\partial_1y_1}{y_1}, \ldots ,\frac{\partial_1y_r}{y_r},\frac{\partial_2y_1}{y_1},\ldots,\frac{\partial_2y_r}{y_r}, \ldots, \frac{\partial_my_1}{y_1}, \ldots,\frac{\partial_my_r}{y_r}\right).$$
The image of $H$ under this homomorphism is a DFGG in $\Ga^{rm}$ and so has differential type at most $0$.  The kernel of this homomorphism restricted to $H$ is $$\left(\Gm\left(\K^\Delta\right)\right)^r \cap H,$$ 
which also has type at most $0$.  Therefore, $\tau(H)\Le 0$.\end{proof}

\begin{lemma}\label{lem4} Let $G$ be  a reductive LDAG.  Then $G$ is a DFGG if and only if  $\tau\big({Z(G)}^\circ\big) \Le 0.$
\end{lemma}
\begin{proof}  Assume that $G$ is a DFGG. By  Proposition~\ref{prop:finite}, we can assume that $G$ is Kolchin-connected as well as a DFGG. From \cite[Theorem~4.7]{diffreductive}, we can assume that $\overline{G}  = P$ is a reductive LAG. From the structure of reductive LAGs, we know that $$P = (P,P) \cdot Z(P),$$ where $Z(P)$ denotes the center, $(P,P)$ is the commutator subgroup and $Z(P) \cap (P,P)$ is finite.  Note also that ${Z(P)}^\circ$ is a torus and that $Z(G)= Z(P)\cap G$.  Let 
$$\pi: P \rightarrow P/(P,P) \simeq Z(P)/[Z(P)\cap(P,P)].$$
The image of $G$ is connected and so lies in $$\pi\big({Z(P)}^\circ\big) \simeq \Gm^t$$ for some $t$. The image is a DFGG and so, by Lemma~\ref{lem3},  must have type at most $0$. From the description of $\pi$, one sees that $$\pi:Z(G) \rightarrow Z(G)/[Z(P)\cap(P,P)] \subset Z(P)/[Z(P)\cap(P,P)].$$ Since $Z(P)\cap(P,P)$ is finite, we have $\tau\big({Z(G)}^\circ\big) \Le 0$. 

Nowadays assume that $\tau\big({Z(G)}^\circ\big) \Le 0.$  \cite[Proposition~2.9]{MiOvSi} implies that $Z(G^\circ)$ is a DFGG.  Therefore, it is enough to show that $G' = G/Z(G)^\circ$ is a DFGG. We see that $G'$ is semisimple, and we will show that any semisimple LDAG is a DFGG.  Clearly, it is enough  to show that this is true under the further assumption that $G'$ is connected.  

Let $\calD$ be the $\K$-vector space spanned by $\Delta$.  \cite[Theorem~18]{CassidyClassification} implies that $G' = G_1\cdot\ldots\cdot G_\ell$, where, for each $i$, there exists a simple LAG $H_i$ defined over $\Q$ and a Lie (A Lie subspace $\calE \subset \calD$ is a subspace such that, for any $\partial, \partial' \in \calE$, we have $\partial\partial' - \partial'\partial \in \calE$.)   $K$-subspace $\calE_i$ of $\calD$ such that $$G_i = H_i\left(\K^{\calE_i}\right),\quad \K^{\calE_i} = \big\{ c\in K \ | \ \partial(c) = 0 \mbox{ for all } \partial \in \calE_i\big\}.$$ Therefore, it suffices to show that, for a simple LAG $H$ and a Lie $\K$-subspace $\calE \subset \calD$, the LDAG $H\big(\K^{\calE}\big)$ is a DFGG.  
From~\cite[Proposition~6 and~7]{KolDAG}, $\calE$ has a $\K$-basis of commuting derivations $\Lambda =\big\{\partial_1', \ldots , \partial_r'\big\}$, which can be extended to a commuting basis $\big\{\partial_1', \ldots , \partial_m'\big\}$ of $\calD$.  Let $\Pi = \big\{\partial_{r+1}', \ldots , \partial_m'\big\}$. \cite[Lemma~9.3]{PhyllisMichael} implies  that $\K^{\calE}$ is differentially closed as a $\Pi$-differential field.  We may consider $H\big(\K^{\calE}\big)$ as a LAG over the $\Pi$-differential field $\K^{\calE}$.  The result now follows from \cite[Lemma~2.2]{MichaelGmGa}.
\end{proof}

\section{Filtrations and gradings of the coordinate ring of an LDAG}\label{sec:maintech}
In this section, we develop the main technique of the paper, filtrations and grading of coordinate rings of LDAGs. Let $\K$ be a $\Delta$-field of characteristic zero, not necessarily differentially closed. 
The set of natural numbers $\{0, 1, 2,\ldots\}$ is denoted by $\N$.

\subsection{Filtrations of $G$-modules}
Let $G$ be an LDAG and $A:=\K\{G\}$ be the corresponding differential Hopf algebra (see \cite[Section~2]{CassidyRep} and \cite[Section~3.2]{OvchRecoverGroup}).  
Fix a faithful $G$-module $W$. Let
\begin{equation}\label{eq:varphi}
\varphi: \K\{\GL(W)\}\to A
\end{equation}
be the differential epimorphism of differential Hopf algebras corresponding to the embedding $G\to\GL(W)$. Set $$H:=\overline{G} ,$$ which is a LAG. Define
\begin{equation}\label{eq:defA0}
A_0:=\varphi(\K[\GL(W)])=\K[H]
\end{equation}
and, for $n\Ge 1$,
\begin{equation}\label{eq:defAn}
A_n:=\Span_\K\left\{\prod_{j\in J}\theta_j y_j \in A\:\Big|\: J\ \text{is a finite set},\  y_j\in A_0,\ \theta_j\in\Theta,\ \sum_{j\in J}\ord(\theta_j)\Le n\right\}.
\end{equation}
 
The following shows that the subspaces $A_n\subset A$ form a filtration (in the sense of~\cite{sweedler}) of the Hopf algebra $A$.
\begin{proposition}\label{prop:PropertiesOfAn}
We have
\begin{align}
 &A=\bigcup_{n\in\N}A_n,\quad A_n\subset A_{n+1},\label{eq:UnionAn}\\
 &A_iA_j\subset A_{i+j},\quad i,j\in\N,\label{eq:ProductAn}\\
 &\Delta(A_n)\subset\sum_{i=0}^nA_i\otimes_\K A_{n-i}.\label{eq:DeltaAn}
\end{align}\qedhere
\end{proposition}

\begin{proof}
Relation~\eqref{eq:ProductAn} follows immediately from~\eqref{eq:defAn}. Since $\K[\GL(W)]$ differentially generates $\K\{\GL(W)\}$ and $\varphi$ is a differential epimorphism, $A_0$ differentially generates $A$, which implies~\eqref{eq:UnionAn}. Finally, let us prove~\eqref{eq:DeltaAn}. Consider the differential Hopf algebra
$$
B:=A\otimes_\K A,
$$
where $\partial_l$, $1\Le l\Le m$, acts on $B$ as follows:
$$
\partial_l(x\otimes y)=\partial_l(x)\otimes y+x\otimes \partial_l(y),\quad x,y\in A.
$$
Set
$$
B_n:=\sum_{i=0}^nA_i\otimes_\K A_{n-i},\quad n\in\N.
$$
We have
\begin{equation}\label{eq:RelationsForBn}
B_iB_j\subset B_{i+j}\qquad\text{and}\qquad\partial_l(B_n)\subset B_{n+1},\quad i,j\in\N,\ n\in\N,\ 1\Le l\Le m.
\end{equation}
Since $\K[\GL(W)]$ is a Hopf subalgebra of $\K\{\GL(W)\}$, $A_0$ is a Hopf subalgebra of $A$. In particular,
\begin{equation}\label{eq:A0ToB0}
\Delta(A_0)\subset B_0.
\end{equation}
Since $\Delta: A\to B$ is a differential homomorphism, definition~\eqref{eq:defAn} and relations~\eqref{eq:A0ToB0},~\eqref{eq:RelationsForBn} imply
$$
\Delta(A_n)\subset B_n,\ n\in\N.\qedhere
$$
\end{proof}

We will call $\{A_n\}_{n\in\N}$ the \emph{$W$-filtration} of $A$. As the definition of $A_n$ depends on $W$, we will sometimes write $A_n(W)$ for $A_n$.  By~\eqref{eq:DeltaAn}, $A_n$ is a subcomodule of $A$. If $x\in A\setminus{A_n}$, then the relation
\begin{equation}\label{eq:UnitRelation}
x=(\epsilon\otimes\Id)\Delta(x)
\end{equation}
shows that $\Delta(x)\not\in A\otimes A_n$. Therefore, $A_n$ is the largest subcomodule $U\subset A$ such that $\Delta(U)\subset U\otimes_\K A_n$. This suggests the following notation.

For $V\in\TRep G$ and $n\in\N$, let $V_n$ denote the largest submodule $U\subset V$ such that
$$
\varrho_V(U)\subset U\otimes_\K A_n.
$$
Then submodules $V_n\subset V$, $n\in\N$, form a filtration of $V$, which we also call the $W$-filtration. 

\begin{proposition}\label{prop:GMorAreFilt}
For a morphism $f: U\to V$ of $G$-modules and an $n\in\N$, we have $f(U_n)\subset V_n$.
\end{proposition}
\begin{proof}
The proof follows immediately from the definition of a morphism of $G$-modules.
\end{proof}

Note that $U_n\subset V_n$ and $V_n\cap U\subset U_n$ for all submodules $U\subset V\in\TRep G$. Therefore,
\begin{align}
& U_n=U\cap V_n\ \ \text{for every subcomodule }\ U\subset V\in\TRep G, \label{eq:Restriction} \\
& (U\oplus V)_n=U_n\oplus V_n\ \ \text{for all}\ U,V\in\TRep G, \label{eq:nSum} \\
& \big(\bigcup\nolimits_{i\in\N} V(i)\big)_n=\bigcup\nolimits_{i\in\N} {V(i)}_n,\quad V(i)\subset V(i+1)\in\TRep G. \label{eq:DirLim}
\end{align}

\begin{proposition}\label{prop:VarrhoV_gen}
For every $V\in\TRep G$, we have
\begin{equation}\label{eq:VarrhoV_gen}
\varrho_V(V_n)\subset\sum_{i=0}^nV_i\otimes_\K A_{n-i}.
\end{equation}
\end{proposition}

\begin{proof}
Let $X$ denote the set of all $V\in\TRep G$ satisfying~\eqref{eq:VarrhoV_gen}. It follows from~\eqref{eq:Restriction} and~\eqref{eq:nSum} that, if $U,V\in X$, then every submodule of $U\oplus V$ belongs to $X$. If $V\in\Rep G$, then $V$ is isomorphic to a submodule of $A^{\dim V}$ by Proposition~\ref{prop:BasicPropsA}. Since $A\in X$ by Proposition~\ref{prop:PropertiesOfAn}, $\Ob(\Rep G)\subset X$. For the general case, it remains to apply~\eqref{eq:DirLim}. 
\end{proof}

Recall that a module is called \emph{semisimple} if it equals the sum of its simple submodules.
\begin{proposition}\label{prop:RedClosure}
Suppose that $W$ is a semisimple $G$-module. Then the LAG $H$ is reductive. If $W$ is not semisimple, then it is not semisimple as an $H$-module.
\end{proposition}

\begin{proof}
For the proof, see~\cite[proof of Theorem~4.7]{diffreductive}.
\end{proof}

\begin{lemma}\label{lem:socV}
Let $V\in\TRep G$. If $V$ is semisimple, then $V=V_0$. (Loosely speaking, this means that all completely reducible representations of an LDAG are polynomial. This was also proved in~\cite[Theorem~3.3]{diffreductive}.) If $W$ is semisimple, the converse is true.
\end{lemma}

\begin{proof}
By~\eqref{eq:nSum}, it suffices to prove the statement for a simple $V\in\Rep G$. Suppose that $V$ is simple and $V=V_n\neq V_{n-1}$. Then $V_{n-1}=\{0\}$, and Proposition~\ref{prop:VarrhoV_gen} implies 
\begin{equation}\label{eq:inVA0}
\varrho_V(V)\subset V\otimes A_0.
\end{equation} Hence, $V=V_0$. 

Suppose that $W$ is semisimple and $V=V_0\in\Rep G$. The latter means~\eqref{eq:inVA0}, that is, the representation of $G$ on $V$ extends to the representation of $H$ on $V$. But $H$ is reductive by Proposition~\ref{prop:RedClosure} (since $W$ is semisimple). Then $V$ is semisimple as an $H$-module. Again, by Proposition~\ref{prop:RedClosure}, the $G$-module $V$ is semisimple.
\end{proof}

\begin{corollary}\label{cor:FiltRed}
If $W$ is semisimple, then $A_0$ is the sum of all simple subcomodules of $A$.
Therefore, if $U, V$ are faithful semisimple $G$-modules, then the $U$- and $V$-filtrations of $A$ coincide.
\end{corollary}
\begin{proof}
By Lemma~\ref{lem:socV}, if $Z \subset A$ is simple, then $Z=Z_0$. Hence, by Proposition~\ref{prop:GMorAreFilt},  $Z$ is contained in $A_0$.  Moreover, by Lemma~\ref{lem:socV}, $A_0$  is the sum of all its simple submodules.
\end{proof}

\begin{corollary}\label{cor:ConnectedGH}
The LDAG $G$ is connected if and only if the LAG $H$ is connected.
\end{corollary}
\begin{proof}
If $G$ is Kolchin connected and $$A = \K\{G\} = \K\{\GL(W)\}/\ip = \K\{X_{ij},1/\det\}/\ip,$$ then the differential ideal $\ip$ is prime \cite[p.~895]{Cassidy}. Since, by \cite[p.~897]{Cassidy}, $$A_0 = \K[H] = \K[\GL(W)]\big/(\ip\cap \K[\GL(W)])= \K[X_{ij},1/\det]\big/(\ip\cap \K[X_{ij},1/\det])$$ and the ideal $\ip\cap \K[X_{ij},1/\det]$ is prime, $H$ is Zariski connected.

Set $\Gamma:=G/G^\circ,$ which is finite. Denote the quotient map by
$$
\pi: G\to\Gamma.
$$ Since $\Gamma$ is finite and $\Char \K =0$, $B:=\K\{\Gamma\}\in\Rep\Gamma$ is semisimple. Then $B$ has a structure of a semisimple $G$-module via $\pi$. Therefore, by Lemma~\ref{lem:socV}, $B=B_0$. Since $\pi^*$ is a homomorphism of $G$-modules, by Proposition~\ref{prop:GMorAreFilt},
$$
\pi^*(B)=\pi^*(B_0)\subset A_0=\K[H].
$$
This means that $\pi$ is a restriction of an epimorphism $H\to\Gamma$, which completes the proof.
\end{proof}

For the $\Delta$-field $\K$, denote the underlying abstract field endowed with the trivial differential structure ($\partial_lk=0$, $1\Le l\Le m$) by $\widetilde{\K}$.

\begin{proposition}\label{prop:IntDomain}
Suppose that the LDAG $G$ is connected. If $x\in A_i$, $y\in A_j$ and $xy\in A_{i+j-1}$, then either $x\in A_{i-1}$ or $y\in A_{j-1}$. 
\end{proposition}
\begin{proof}
We need to show that the graded algebra
$$
\gr A:=\bigoplus_{n\in\N} A_n/A_{n-1}
$$
is an integral domain. Note that $\gr A$ is a differential algebra via
$$
\partial_l(x+A_{n-1}):=\partial_l(x)+A_{n},\quad x\in A_n.
$$
Furthermore, to a homomorphism $\nu: B\to C$ of filtered algebras such that $\nu(B_n)\subset C_n$, $n\in\N$, there corresponds the homomorphism
$$
\gr\nu:\gr B\to\gr C,\quad x+B_{n-1}\mapsto \nu(x)+C_{n-1},\quad x\in B_n.
$$
Let us identify $\GL(W)$ with $\GL_d$, $d:=\dim W$, and set
$$
B:=\Q\big\{x_{ij},1/\det\big\},
$$
the coordinate ring of $\GL_d$ over $\Q$. The algebra $B$ is graded by
$$
\overline{B}_n:=\Span_{\Q}\left\{\prod_{j\in J}\theta_j y_j \:\Big|\: J\ \text{is a finite set},\  y_j\in \Q[\GL_d],\ \theta_j\in\Theta,\ \sum_{j\in J}\ord(\theta_j)=n\right\},\quad n\in\N.
$$ 
The $W$-filtration of $B$ is then associated with this grading:
$$
B_n=\bigoplus_{i=0}^n\overline{B}_i.
$$
For a field extension $\Q\subset L$, set $_LB:=B\otimes_{\Q} L$, a Hopf algebra over $L$. Then the algebra $_LB$ is graded by $_L\overline{B}_n:=\overline{B}_n\otimes L$.

Let $I$ stand for the Hopf ideal of $_\K B$ defining $G\subset\GL_d$. For $x\in{} _\K B$, let $x_h$ denote the highest degree component of $x$ with respect to the grading $\big\{{}_\K\overline{B}_n\big\}$. Let $\widetilde{I}$ denote the $\K$-span of all $x_h$, $x\in I$. 
As in the proof of Proposition~\ref{prop:PropertiesOfAn}, we conclude that, for all $n \in \N$, 
\begin{equation}\label{eq:Bnbar}
\Delta\big(\overline{B}_n\big)\subset \sum_{i=0}^n\overline{B}_i\otimes_\K\overline{B}_{n-i}. 
\end{equation}
Since $\Delta(I) \subset I\otimes_\K B+B\otimes_\K I$, inclusion~\eqref{eq:Bnbar} implies that, for all $n\in\N$ and $x \in I\cap B_n$, 
$$
I\otimes_\K B_n+B_n\otimes_\K I \ni \Delta(x)= \Delta(x-x_h)+\Delta(x_h)\in \left(\sum_{i=0}^{n-1}B_{i}\otimes_\K B_{n-i-1}\right)\oplus\left(\sum_{i=0}^n\overline{B}_i\otimes_\K\overline{B}_{n-i}\right).
$$
Hence,  by induction, one has$$
\Delta(x_h) \in\widetilde{I}\otimes_\K B_n+B_n\otimes_\K\widetilde I \subset \widetilde I\otimes_\K B+B\otimes_\K\widetilde I.
$$
We have $S(I)\subset I$, where $S: B\to B$ is the antipode. Moreover, since $S(B_0)=B_0$ and $S$ is differential,
$$
S\big(\overline{B}_n\big) \subset \overline{B}_n,\ n\in\N.
$$
Hence,
$$
S(x_h) = S(x_h-x+x)=S(x_h-x)+S(x)\in (B_{n-1}+I)\cap \overline{B}_n,
$$
which implies that $$S\big(\widetilde I\big)\subset\widetilde I.$$
Therefore, $\widetilde{I}$ is a Hopf ideal of $_\K B$ (not necessarily differential!). Consider the algebra map
$$
\alpha:{}_\K B\stackrel{\beta}{\simeq}\gr _\K B\stackrel{\gr\varphi}{\longrightarrow}\gr A,
$$
where $\beta$ is defined by the sections $$_\K \overline{B}_n\to{} _\K B_n\big/{}_\K B_{n-1},\quad n\in\N,$$ and $\varphi$ is given by~\eqref{eq:varphi}. 
For every $x \in I$, let $n\in\N$ be such that $x_h \in \overline{B}_n$. Then $$\varphi(x_h)=\varphi(x_h-x+x)=\varphi(x_h-x)+\varphi(x)=\varphi(x_h-x)+0\in A_{n-1}.$$
Hence, $$\widetilde{I} \subset \Ker\alpha.$$ On the other hand, let $\alpha(x)=0$. Then there exists $n \in \N$ such that, for all $i$, $0\Le i\Le n$, if $x_i \in \overline{B}_i$ satisfy $\beta(x) = x_0+\ldots+x_n$, then
$$
\varphi(x_i) \in A_{i-1},$$
which implies that there exists $y_i \in I\cap B_i$ such that
$$
x_i -y_i\in B_{i-1}.
$$
Therefore, $\beta^{-1}(x_i) \in \widetilde{I}$, implying that $$\Ker\alpha\subset\widetilde{I}.$$
Thus, $\alpha$ induces a Hopf algebra structure on $\gr A$. (In general, if $A$ is a filtered Hopf algebra, then $\gr A$ can be given (in a natural way) a structure of a graded Hopf algebra; see, e.g.,~\cite[Chapter~11]{sweedler}.)  Consider the identity map (This map is differential if and only if $\K $ is constant.)
$$
\gamma:{} _{\widetilde{\K}}B\to{} _\K B
$$
of Hopf algebras. Since
$$
\gamma\big({}_{\widetilde{\K}}\overline{B}_n\big)={}_\K \overline{B}_n,
$$
$J:=\gamma^{-1}\big(\widetilde{I}\big)$ is a Hopf ideal of $_{\widetilde{\K }}B$. Moreover, it is differential, since 
$$
\partial_l\big(x_h\big)=\big(\partial_lx\big)_h,\ x\in{} _{\widetilde{\K }}B.
$$
Therefore, $\gr A$ has a structure of a differential Hopf algebra over~$\widetilde{\K }$. Furthermore it is differentially generated by the Hopf algebra $A_0\subset\gr A$. In other words, $\gr A$ is isomorphic to the coordinate algebra of an LDAG $\widetilde{G}$ (over $\widetilde{\K }$) dense in $H$. By Corollary~\ref{cor:ConnectedGH}, $\widetilde{G}$ is connected. Hence, $\gr A$ has no zero divisors.
\end{proof}

\subsection{Subalgebras generated by $W$-filtrations}
For $n\in\N$, let $A_{(n)}\subset A$ denote the subalgebra generated by $A_n$. Since $A_n$ is a subcoalgebra of $A$, it follows that $A_{(n)}$ is a Hopf subalgebra of $A$. Note that $\big\{A_{(n)},\ n\in\N\big\}$ forms a filtration of the vector space $A$. We will prove the result analogous to Proposition~\ref{prop:IntDomain}.

\begin{proposition}\label{prop:SubalgFilt}
Suppose that $G$ is connected. If $x\in A_{(n)}$, $y\in A_{(n+1)}$, and $xy\in A_{(n)}$, then $y\in A_{(n)}$.
\end{proposition}

\begin{proof}
Let $G_n$, $n\in\N$, stand for the LAG with the (finitely generated) Hopf algebra $A_{(n)}$. Since $A_{(n)}\subset A$ and $A$ is an integral domain, $A_{(n)}$ is an integral domain.  Let $G_{n+1}\to G_{n}$ be the epimorphism of LAGs that corresponds to the embedding $A_{(n)}\subset A_{(n+1)}$ and $K$ be its kernel. Then we have
$$
A_{(n)}=A_{(n+1)}^K.
$$
Denote $A_{(n+1)}$ by $B$. We have $$x\in B^K,\ \ y\in B,\ \ \text{and} \ \ xy\in B^K.$$ Let us consider this relation in $\Quot B\supset B$. We have
$$
y\in(\Quot B)^K\cap B=B^K.
$$
Thus, $y\in A_{(n)}$.
\end{proof}

For $s,t\in\N$, set
$$
A_{s,t}:=A_s\cap A_{(t)}.
$$
Since $A_n\subset A_{(n)}$, $A_{s,t}=A_s$ if $s\Le t$. Also, $A_{s,0}=A_0$ for all $s\in\Z_+$. Therefore, one may think of $A_{s,t}$ as a filtration of the $G$-module $V$, where the indices are ordered by the following pattern:
\begin{equation}\label{eq:OrderIndice}
(0,0)=0<(1,1)=1<(2,1)<(2,2)=2<(3,1)<(3,2)<\ldots.
\end{equation}
(Note that $t=0$ implies $s=0$.)
We also have
\begin{equation}\label{eq:ProdAst}
A_{s_1,t_1}A_{s_2,t_2}\subset A_{s_1+s_2,\max\{t_1, t_2\}}
\end{equation}

\begin{theorem}\label{thm:ProductInA}
Let $x_i\in A$, $1\Le i\Le r$, and $x:=x_1x_2\cdot\ldots\cdot x_r\in A_{s,t}$. Then, for all $i$, $1\Le i\Le r$, there exist $s_i,\, t_i\in\N$ such that $x_i\in A_{s_i, t_i}$ and 
$$
\sum_is_i\Le s\quad\text{and}\quad\max_i\{t_i\}\Le t.\qedhere
$$
\end{theorem}

\begin{proof}
It suffices to consider only the case $r=2$. Then, Propositions~\ref{prop:IntDomain} and~\ref{prop:SubalgFilt} complete the proof.
\end{proof}

For $V\in\TRep G$ and $n\in\N$, let $V_{(n)}$ denote the largest submodule $U$ of $V$ such that $\varrho_V(U)\subset U\otimes A_{(n)}$. (If $V=A$, then $V_{(n)}=A_{(n)}$, which follows from~\eqref{eq:UnitRelation}.) Similarly, we define $V_{s,t}$, $s,t\in\N$.

For a reductive LDAG $G$ and its coordinate ring $A=\K \{G\}$,  let ${\{A_n\}}_{n\in\N}$  denote  the $W$-filtration corresponding to an arbitrary faithful semisimple $G$-module $W$. This filtration does not depend on the choice of $W$ by Corollary~\ref{cor:FiltRed}. 

\begin{definition}\label{def:GV}
If $\phi: G\to L$ is a homomorphism of LDAGs and $V\in\TRep L$, then $\phi$ induces the structure of a $G$-module on $V$. This $G$-module will be denoted by $_GV$.
\end{definition}

\begin{proposition}\label{prop:QuasiIsom}
Let $\phi: G\to L$ be a homomorphism of reductive LDAGs.  Then
\begin{equation}\label{eq:phi*}
\phi^*\big(B_{s,t}\big)\subset A_{s,t},\quad s,t\in\N,
\end{equation}
where $A:=\K \{G\}$ and $B:=\K \{L\}$.
Suppose that $\Ker\phi$ is finite and the index of $\phi(G)$ in $L$ is finite. Then, for every $V\in\TRep L$, 
\begin{equation}\label{eq:phi*^-1}
V=V_{s,t}\ \Longleftrightarrow\ _GV={(_GV)}_{s,t},\quad s,t\in\N.
\end{equation}
\end{proposition}

\begin{proof}
Applying Lemma~\ref{lem:socV} to $V:=B_0$ and Proposition~\ref{prop:GMorAreFilt} to $\phi^*$, we obtain $\phi^*(B_0)\subset A_0$. Since $\phi^*$ is a differential homomorphism, relation~\eqref{eq:phi*} follows.

Let us prove the second statement of the Proposition. Note that the implication $\Rightarrow$ of~\eqref{eq:phi*^-1} follows directly from~\eqref{eq:phi*}. We will prove the implication $\Leftarrow$. It suffices to consider two cases:
\begin{enumerate}
\item $G$ is connected and $\phi$ is injective;
\item $G$ is connected and $\phi$ is surjective;
\end{enumerate} 
which follows from the commutative diagram
$$
\begin{CD}
G^\circ @>\phi|_{G^\circ} >> L^\circ\\
@VVV @VVV\\
G @>\phi >> L.
\end{CD}
$$ 
Moreover, by~\eqref{eq:DirLim} and Proposition~\ref{prop:BasicPropsA}, it suffices to consider the case of finite-dimensional $V$. By the same proposition, there is an embedding of $L$-modules
$$
\eta: V\to B^{d}, \ d:=\dim V.
$$
Then $_GV$ is isomorphic to $\phi^*_d\eta(V)$, where $\phi^*_d: B^{d}\to A^d$ is the application of $\phi^*$ componentwise. If $_GV={(_GV)}_{s,t}$, then $\phi^*_d\eta(V)\subset A_{s,t}^d$. Hence, setting $V(i)$ to be the projection of $\eta(V)$ to the $i$th component of $B^d$, we conclude $\phi^*(V(i))\subset A_{s,t}$ for all $i$, $1\Le i\Le d$. If we show that this implies $V(i)\subset B_{s,t}$, we are done. So, we will show that, if $V\subset B$, then 
$$
\phi^*(V)=\phi^*{(V)}_{s,t}\Longrightarrow V=V_{s,t}.
$$
\textbf{Case (i)}. Let us identify $G$ with $L^\circ$ via $\phi$. Suppose $L\subset\GL(U)$, where $U$ is a semisimple $L$-module. Let $g_1=1,\ldots,g_r\in L$ be representatives of the cosets of  $L^\circ$. Let $I(j)\subset B$, $1\Le j\Le r$, be the differential ideal of functions vanishing on all connected components of $L$ but $g_jL^\circ$. We have
$$
B=\bigoplus_{j=1}^rI(j)\qquad\text{and}\qquad I(j)=g_jI(1).
$$
The $G$-modules $I:=I(1)$ and $A$ are isomorphic, and the projection $B\to I$ corresponds to the restriction map $\phi^*$. The $G$-module structure on $I(j)$ is obtained  by the twist by conjugation $G\to G$, $g\mapsto g_j^{-1}gg_j$. Since a conjugation preserves the $U$-filtration of $B$, we conclude
$$
g_j(I_n)=\big(g_jI\big)_n.
$$
By Corollary~\ref{cor:ConnectedGH}, Zariski closures of connected components of $L\subset\GL(U)$ are connected components of $\overline{L} $. Therefore,
$$
B_0=\bigoplus_{j=1}^rg_j(I_0).
$$
Then $B_0\cap I=I_0$. Since $I$ is a differential ideal, $B_n\cap I=I_n$ for all $n\in\N$.
Let 
\begin{equation}\label{eq:vj}
v\in V_n\setminus{V_{n-1}}.
\end{equation} Then, for each $j$, $1\Le i \Le r$, there exists $v(j) \in I(j)$ such that
$$
v=\sum_{j=1}^rv(j).
$$
By~\eqref{eq:vj},
there exists $j$, $1\Le j\Le r$, such that $v(j)\in V_n\setminus{V_{n-1}}$. Set $$w:=g_j^{-1}v\in V_n\setminus{V_{n-1}}.$$ Then, by the above, $$\phi^*(w)\in A_n\setminus A_{n-1}.$$ We conclude that, for all $n\in\N$, $$\phi^*(V)=\phi^*(V)_n \quad\Longrightarrow\quad V=V_n.$$
Similarly, one can show that $$\phi^*(V)=\phi^*(V)_{(n)} \quad\Longrightarrow\quad V=V_{(n)}.$$ Since $V_{s,t}=V_s\cap V_{(t)}$, this completes the proof of Case (i).

\textbf{Case (ii)}.
Consider $B$ as a subalgebra of $A$ via $\phi^*$. It suffices to show
\begin{equation}\label{eq:Preimage}
A_{s,t}\cap B\subset B_{s,t}.
\end{equation} 
We have $B\subset A^{\Gamma}$, where $\Gamma:=\Ker\phi$. 

Let us show that $B_0=A_0^\Gamma$. For this, consider $G$ and $L$ as differential algebraic Zariski dense subgroups of reductive
LAGs. Since $B_0\subset A_0$, the map $\phi$ extends to an epimorphism
$$
\overline\phi:\overline{G} \to \overline{L} .
$$ 
Since $\overline\Gamma =\Gamma$, $\Gamma$ is normal in $\overline{G} $. Hence, $\overline\phi$ factors through the epimorphism
$$
\mu:\overline{G} /\Gamma\to\overline{L} .
$$
If $K$ is the image of $G$ in the quotient $\overline{G} /\Gamma$, then $\mu(K)=L$ and $\mu$ is an isomorphism on $K$. This means that $\mu^*$ extends to an isomorphism of $B=\K\{L\}$ onto $\K\{K\}$. Since $K$ is reductive, the isomorphism preserves the grading by the first part of the proposition. In particular, $\mu^*(B_0)=\K\{K\}_0$. As $K$ is dense in $\overline{G} /\Gamma$, we obtain
$$
B_0=\K \big[\overline{L} \big]=\K \big[\overline{G} /\Gamma\big]=\K \big[\overline{G} \big]^\Gamma=A_0^\Gamma.
$$
  Let us consider the following sets:
$$
\widetilde{A}_{s,t}:=\big\{x\in (A_{s,t})^\Gamma\:\big|\:\exists\, 0\ne b\in B_0\, :\,  bx\in B_{s,t}\big\},\quad s,t\in\N.
$$
These are $B_0$-submodules of $A$ (via multiplication) satisfying~\eqref{eq:ProdAst}, as one can check. Moreover, for every $l$, $1\Le l\Le m$,
\begin{equation}\label{eq:Partial_st}
\partial_l\big(\widetilde{A}_{s,t}\big)\subset\widetilde{A}_{s+1,t+1}.
\end{equation}
Indeed, let $x\in\widetilde{A}_{s,t}$, $b\in B_0$, and $bx\in B_{s,t}$. Then
$$
b^2\partial_l(x)=b(\partial_l(bx)-x\partial_l(b))=b\partial_l(bx)-(bx)\partial_l(b)\in B_{s+1,t+1}.
$$   
Hence, $$\partial_l(x)\in\widetilde{A}_{s+1,t+1}.$$
We have
$$
B_{s,t}\subset \widetilde{A}_{s,t}\subset \big(A_{s,t}\big)^\Gamma.
$$
We will show that
\begin{equation}\label{eq:TildeFilt}
\widetilde{A}_{s,t}=\big(A_{s,t}\big)^\Gamma.
\end{equation}
This will complete the proof as follows. Suppose that $$x\in B\cap A_{s,t}\subset \big(A_{s,t}\big)^\Gamma.$$ By~\eqref{eq:TildeFilt}, there exists $b\in B_0$ such that $bx\in B_{s,t}$. Then, Theorem~\ref{thm:ProductInA} implies $x\in B_{s,t}$. We conclude~\eqref{eq:Preimage}.

Now, let us prove~\eqref{eq:TildeFilt} by induction on $s$, the case $s=0$ being already considered above. Suppose, $s\Ge 1$. 
Since $\Gamma$ is  a finite normal subgroup of the connected group $\overline{G}$, it is commutative \cite[Lemma~V.22.1]{Borel}. Therefore,  every $\Gamma$-module has a basis consisting of semi-invariant vectors, that is, spanning $\Gamma$-invariant $\K $-lines. Therefore, since a finite subset of the algebra $A_0$ belongs to a finite-dimensional subcomodule and $A_0$ is finitely generated, one can choose $\Gamma$-semi-invariant generators $X:=\{x_1,\ldots, x_r\}\subset A_0$ of $A$. Note that $X$ differentially generates $A$. Since $\Gamma$ is finite, its scalar action is given by algebraic numbers, which are constant with respect to the derivations of $\K $. Hence, the actions of $\Gamma$ and $\Theta$ on $A$ commute, and an arbitrary product of elements of the form $\theta x_i$, $\theta\in\Theta$, is $\Gamma$-semi-invariant.

Let $0\neq x\in(A_{s,t})^\Gamma$. We will show that $x\in \widetilde{A}_{s,t}$.  Since a sum of $\Gamma$-semi-invariant elements is invariant if and only if each of them is invariant, it suffices to consider the case
\begin{equation}\label{eq:xyj}
x=\prod_{j\in J}\theta_j y_j,\ \theta_j\in\Theta,
\end{equation}
where $J$ is a finite set and $y_j\in X\subset A_0$. Moreover, by Theorem~\ref{thm:ProductInA},~\eqref{eq:xyj} can be rewritten to satisfy
$$
\sum_{j\in J}\ord\theta_j\Le s\quad\text{and}\quad\max_{j\in J}\big\{\ord\theta_j\big\}\Le t.
$$ 
Since $y_j$ and $\theta_jy_j$ have the same $\Gamma$-weights,
$$
y:=\prod_{j\in J}y_j\in (A_0)^\Gamma= B_0.
$$ 
Set $g:=|\Gamma|$. We have
$$
y^{g-1}x=\prod_{j\in J}y_j^{g-1}\theta_j(y_j)\in \big(A_{s,t}\big)^\Gamma
$$
and, for every $j\in J$, $$y_j^{g-1}\theta_j(y_j)\in \big(A_{\ord\theta_j}\big)^\Gamma.$$ 

If $\ord\theta_j<s$ for all $j\in J$, then, by induction, $$y_j^{g-1}\theta_j(y_j)\in\widetilde{A}_{\ord\theta_j,\ord\theta_j}$$ for all $j\in J$. This implies  $$y^{g-1}x\in\widetilde{A}_{s,t}.$$ Hence, $x\in\widetilde{A}_{s,t}$.

Suppose that there is a $j\in J$ such that $\ord\theta_j=s$. Let us set $\theta:=\theta_j$. Then, there exist $i$, $1\Le i\Le r$, and $a\in A_0$ such that
$$
x=a\theta(x_i)\in A_s^\Gamma. 
$$
It follows that
$$
ax_i\in A_0^\Gamma=B_0.
$$
We will show that $x\in\widetilde{A}_{s,s}=:\widetilde{A}_s$.
There exist $l$, $1\Le l\Le m$, and $\widetilde\theta\in\Theta$, $\ord\widetilde\theta=s-1$, such that $$\theta=\partial_l\widetilde\theta.$$ If $s=1$, then $\theta=\partial_l$ and
$$
x_i^gx=(ax_i)\big(x_i^{g-1}\partial_lx_i\big)=(ax_i)\partial_l\big(x_i^g\big)/g\in B_1\subset \widetilde{A}_1,
$$
since $x_i^g\in B_0$. Therefore, $x\in\widetilde{A}_1$.
Suppose that $s\Ge 2$. We have
$$
x=\partial_l\big(a\widetilde\theta(x_i)\big)-\partial_l(a)\widetilde\theta(x_i).
$$
Since $u:=a\widetilde\theta(x_i)\in (A_{s-1})^\Gamma$, by induction, $u\in\widetilde{A}_{s-1}$. Hence, $$\partial_l(u)\in\widetilde{A}_s.$$ Since $s\Ge 2$, we have $$1=\ord\partial_l <s\quad\text{and}\quad \ord\widetilde\theta<s.$$ Since 
$$
v:=\partial_l(a)\widetilde\theta(x_i)=x-\partial_l(u)\in A_s^\Gamma,
$$
by the above argument (for dealing with the case $\ord\theta_j<s$ for all $j\in J$), $v\in\widetilde{A}_{s}.$  Therefore, $$x=\partial_l(u)-v\in\widetilde{A}_s.\qedhere $$
\end{proof}

\section{Filtrations of $G$-modules in reductive case}\label{sec:main:bound} 
In this section, we show our main result, the bounds for differential representations of semisimple LDAGs (Theorem~\ref{thm:EqualFiltrations}) and reductive LDAGs with $\tau(Z(G^\circ)) \Le 0$  (Theorem~\ref{thm:bound}; note that  Lemma~\ref{lem4} implies that, if $\K$ is differentially closed, then  a reductive DFGG has this property).
In particular, we show that, if $G$ is a semisimple LDAG, $W$ is a faithful semisimple $G$-module, and $V\in\Rep G$, then the $W$-filtration of $V$ coincides with its socle filtration. 

\subsection{Socle of a $G$-module}
Let $G$ be an LDAG. Given a $G$-module $V$, its \emph{socle} $\soc V$ is the sum of all simple submodules of $V$. The ascending filtration $\{\soc^nV\}_{n\in\N}$ on $V$ is defined by
$$
\soc^nV\big/\soc^{n-1}V=\soc\big(V\big/\soc^{n-1}V\big),\quad\text{where } \soc^0V:=\{0\}\ \ \text{and}\ \  \soc^1V:=\soc V.
$$ 
\begin{proposition}\label{prop: PropsOfSoc}
Let $n\in\N$.
\begin{enumerate}
 \item If $\varphi: V\to W$ is a homomorphism of $G$-modules, then 
 \begin{equation}\label{eq:1}
 \varphi(\soc^nV)\subset\soc^nW.
 \end{equation}
 \item If $U, V\subset W$ are $G$-modules and $W=U+V$, then 
 \begin{equation}\label{eq:2}
 \soc^nW=\soc^nU+\soc^nV.
 \end{equation}
 \item If $V\in\Rep G$, then 
 \begin{equation}\label{eq:3}
 \soc^n\big(P_1^{i_1}\cdot\ldots\cdot P_m^{i_m}(V)\big)\subset P_1^{i_1}\cdot\ldots\cdot P_m^{i_m}\big(\soc^nV\big).
 \end{equation}
\end{enumerate}
\end{proposition}

\begin{proof}
Let $\varphi: V\to W$ be a homomorphism of $G$-modules. Since the image of a simple module is simple,
$$
\varphi(\soc V)\subset\soc W.
$$
Suppose by induction that
$$
\varphi\big(\soc^{n-1} V\big)\subset\soc^{n-1} W.
$$
Set $\bar V:=V\big/\soc^{n-1}V$, $\bar W:=W\big/\soc^{n-1}W$. We have the commutative diagram:
$$
\begin{CD}
V @>\varphi >> W\\
@VV\pi_{V} V @VV\pi_{W} V\\
\bar V @>\bar\varphi >> \bar W,
\end{CD}
$$
where $\pi_V$ and $\pi_W$ are the quotient maps. Hence,
$$
\varphi\big(\soc^nV\big)\subset\pi_W^{-1}\bar\varphi\pi_V\big(\soc^nV\big)=\pi_W^{-1}\bar\varphi\big(\soc \bar V\big)\subset\pi_W^{-1}\soc\bar W=\soc^nW,
$$
where we used $\bar\varphi\big(\soc\bar V\big)\subset\soc\bar W$.
Let us prove~\eqref{eq:2}. Let $U, V\subset W$ be $G$-modules. It follows immediately from the definition of the socle that
$$
\soc (U+V)=\soc U+\soc V.
$$
Note that, by~\eqref{eq:1}, $V\cap \soc^nW = \soc^nV$.
We have
$$
W/\soc^nW=\big(U\big/\soc^nW\big)+\big(V\big/\soc^nW\big)=\big(U\big/\soc^nU\big)+\big(V\big/\soc^nV\big).
$$
Applying $\soc$, we obtain statement~\eqref{eq:2}.

In order to prove~\eqref{eq:3}, it suffices to do it only for $P_i(V)$, since the other cases would follow by induction. 
Let
$$
\pi_i: P_i(V)\to V
$$
be the natural epimorphism from~\eqref{eq:es}. We have $\pi_i^{-1}(U)=P_i(U)+V$  for all submodules $U\subset V$. Hence, by~\eqref{eq:1},
$$
\soc^n P_i(V)\subset\pi_i^{-1}\big(\soc^nV\big)=P_i\big(\soc^nV\big)+V.
$$
Since $\soc^n\soc^nM=\soc^nM$ for an arbitrary module $M$,
$$
\soc^n P_i(V)=\soc^n\soc^n P_i(V)\subset \soc^n\big(P_i\big(\soc^nV\big)+V\big)\subset P_i\big(\soc^n V\big)+\soc^nV= P_i\big(\soc^n V\big).\qedhere
$$
\end{proof}

\begin{proposition}\label{prop:SocTensor}
Suppose that
$$
\soc(U\otimes V)=(\soc U)\otimes(\soc V)
$$
for all $U, V\in\Rep G$. Then
\begin{equation}\label{eq:SocTensor}
\soc^n(U\otimes V)=\sum_{i=1}^n\big(\soc^iU\big)\otimes\big(\soc^{n+1-i}V\big)
\end{equation}
for all $U, V\in\Rep G$ and $n\in\N$.
\end{proposition}

\begin{proof}
For a $G$-module $V$, denote $\soc^nV$ by $V^n$, $n\in\N$. Suppose by induction that~\eqref{eq:SocTensor} holds for all $n\Le p$ and $U, V\in\Rep G$. Set
$$
S_p=S_p(U,V):=\sum_{i=1}^pU^i\otimes V^{p+1-i}.
$$
For all $1\Le i\Le p$, we have
$$
F_i:=\big(U^i\otimes V^{p+2-i}\big)\big/\big(S_p\cap \big(U^i\otimes V^{p+2-i}\big)\big)=\big(U^i\otimes V^{p+2-i}\big)\big/\big(U^{i-1}\otimes V^{p+2-i} + U^i\otimes V^{p+1-i}\big).
$$
Hence,
$$
F_i\simeq \big(U^i\big/U^{i-1}\big)\otimes \big(V^{p+2-i}\big/V^{p+1-i}\big).
$$
By the hypothesis, $F_i$ is semisimple. Hence, so is
$$
S_{p+1}/S_p=\sum_{i=1}^pF_i\subset (U\otimes V)/S_p.
$$
By the inductive hypothesis, we conclude
$$
\soc^{p+1}(U\otimes V)\supset S_{p+1}.
$$
Now, we prove the other inclusion. Let $$\psi:U\to \bar U:=U/U^1$$ be the quotient map. Note the commutative diagram
$$
\begin{CD}
U\otimes V @>\pi >> X:=(U\otimes V)\big/S_p\\
@VV\psi\otimes\Id V @VVV\\
\bar U\otimes V @>\bar\pi >> \bar X:=\big(\bar U\otimes V\big)\big/S_{p-1}\big(\bar U, V\big),
\end{CD}
$$
where $\pi$ and $\bar\pi$ are the quotient maps. By the inductive hypothesis, we have
$$
\soc^{p+1}(U\otimes V)=\pi^{-1}\big(X^1\big)\subset (\bar\psi\otimes\Id)^{-1}\big(\bar\pi^{-1}\big)\big(\soc \bar X\big)=(\bar\psi\otimes\Id)^{-1}\big(\soc^p\big(\bar U\otimes V\big)\big)\subset S_{p+1},
$$
since $\psi^{-1}\big(\soc^i\bar U\big)=\soc^{i+1}U$.
\end{proof}

It is convenient sometimes to consider the Zariski closure $H$ of $G\subset\GL(W)$ as an LDAG. To  distinguish the structures, let us denote the latter by $H^{\diff}$. Then $\Rep H^{\diff}$ is identified with a subcategory of $\Rep G$. 
\begin{lemma}\label{lem:HTensor}
If  $H$ is reductive, then~\eqref{eq:SocTensor} holds for all $U,V\in\Rep H^{\diff}$ and $n\in\N$. 
\end{lemma}

\begin{proof}
By Proposition~\ref{prop:SocTensor}, we only need to prove the formula for $n=1$. Since $A_0^2=A_0$, we have, by Lemma~\ref{lem:socV},
$$
(\soc U)\otimes(\soc V)=U_0\otimes V_0\subset {(U\otimes V)}_0=\soc(U\otimes V).
$$
Let us prove the other inclusion. Since $\Char \K =0$, $$\soc (U\otimes_\K L)=(\soc U)\otimes_\K L$$ for all differential field extensions $L\supset \K $ by~\cite[Section~7]{Bourbaki8}. Therefore, without loss of generality, we will assume that $\K $ is algebraically closed. Moreover, by Lemma~\ref{lem:socV} and Proposition~\ref{prop:QuasiIsom}, an $H^{\diff}$-module is semisimple if and only if it is semisimple as an $\big(H^{\diff}\big)^\circ$-module. Therefore, it suffices to consider only the case of connected $H$. Since a connected reductive group over an algebraically closed field is defined over $\Q$ and the defining equations of $H^{\diff}$ are of order $0$, the $W$-filtration of $B:=\K \big\{H^{\diff}\big\}$ is associated with a grading (see proof of Proposition~\ref{prop:IntDomain}). In particular, the sum $I$ of all grading components but $B_0=\K [H]$ is an ideal of $B$. We have
$$
B=B_0\oplus I.
$$
Since $B$ is an integral domain, it follows that, if $x,y\in B$ and $xy\in B_0$, then $x,y\in B_0$. Hence,
$$
(U\otimes V)_0\subset U_0\otimes V_0,
$$
which completes the proof.
\end{proof}

\begin{proposition}\label{prop:EasyInlusion}
For all $V\in\TRep G$,
$$
V_n\subset\soc^{n+1}V.\qedhere
$$
\end{proposition}
\begin{proof}
We will use induction on $n\in\N$, with the case $n=0$ being done by Lemma~\ref{lem:socV}. Suppose $n\Ge 1$ and $$V_{n-1}\subset\soc^nV.$$ We need to show that the $G$-module
$$
W:=\big(V_n+\soc^nV\big)\big/\soc^nV\simeq V_n\big/\big(V_n\cap\soc^nV\big)
$$
is semisimple. But the latter is isomorphic to a quotient of $U:=V_n/V_{n-1}$, since $$V_{n-1}\subset V_n\cap\soc^nV.$$ By Proposition~\ref{prop:VarrhoV_gen}, $U=U_0$. Finally, Lemma~\ref{lem:socV} implies that $U$, hence, $W$, is semisimple.
\end{proof}

\subsection{Main result for semisimple LDAGs}
\begin{theorem}\label{thm:EqualFiltrations}
If $G^\circ$ is semisimple, then, for all $V\in\TRep G$ and $n\in\N$,
$$
V_n=\soc^{n+1}V.
$$
\end{theorem}
\begin{proof}
By Proposition~\ref{prop:EasyInlusion}, it suffices to prove that, for all $V\in\Rep G$ and $n\in\N$,
\begin{equation}\label{eq:HardInclusion}
\soc^{n+1}V\subset V_n.
\end{equation}
Let $X\subset\Ob(\Rep G)$ denote the family of all $V$ satisfying~\eqref{eq:HardInclusion} for all $n\in\N$. We have, by Lemma~\ref{lem:socV}, $V\in X$ for all semisimple $V$. Suppose that $V,W\in\Rep H^{\diff}\subset\Rep G$ belong to $X$. Then  $V\oplus W$ and $V\otimes W$ belong to $X$. Indeed, by Propositions~\ref{prop:VarrhoV_gen} and~\ref{prop: PropsOfSoc} and Lemma~\ref{lem:HTensor},
$$
\soc^{n+1}(V\oplus W)=\soc^{n+1}V\oplus\soc^{n+1}W\subset V_n\oplus W_n=(V\oplus W)_n
$$
and
$$
\soc^{n+1}(V\otimes W)=\sum_{i=0}^{n}\big(\soc^{i+1}V\big)\otimes\big(\soc^{n+1-i}W\big)\subset\sum_{i=0}^{n}V_i\otimes W_{n-i}\subset {(V\otimes W)}_n.
$$
Similarly, Proposition~\ref{prop: PropsOfSoc} and~\eqref{eq:Restriction} imply  that, if $V\in X$, then all possible submodules and differential prolongations  of $V$ belong to $X$. Since $\Rep G$ is differentially generated by a semisimple $V\in\Rep H$, it remains only to check the following.

 If $V\in\Rep G$ satisfies~\eqref{eq:HardInclusion}, then so do the dual $V^\vee$ and a quotient $V/U$, where $U\in\Rep G$. Since $G^\circ$ is semisimple, \cite[Theorem~18]{CassidyClassification} implies that $G^\circ(\U)$, $\U$ a differentially closed field containing $\K$, is differentially isomorphic to a group of the form $G_1\cdot G_2\cdot \ldots \cdot G_t$ where, for each $i$, there is an  algebraically closed field $\U_i$ such that $G_i$ is differentially  isomorphic to the $\U_i$ points of a simple algebraic group $H_i$.  Since $H_i = [H_i,H_i]$, we have $G^\circ = [G^\circ,G^\circ]$ and so we must have $G^\circ\subset\SL(V)$. The group $\SL(V)$ acts on $V^{\otimes\dim V}$ and has a nontrivial invariant element  corresponding to the determinant. We conclude that, for $$r:=|G/G^\circ|\dim V,$$ the $\SL(V)$-module $V^{\otimes r}$ has a nontrivial $G$-invariant element. Let $E\subset\GL(V)$ be the group generated by $\SL(V)$ and $G$. Then the space
\begin{equation}\label{eq:home}
\Hom_E\big(V^\vee,V^{\otimes r-1}\big)\simeq \big(V^{\otimes r}\big)^E
\end{equation}
is nontrivial. Since $V^\vee$ is a simple $E$-module, this means that there exists an embedding $$V^\vee\to V^{\otimes r-1}$$ of $E$-modules, and hence of $G$-modules. Then $V^\vee\in X$. Finally, since $(V/U)^\vee$ embeds into $V^\vee$, it belongs to $X$. Then its dual $V/U\in X$. Hence, $X=\Ob(\Rep G)$.
\end{proof}

\subsection{Reductive case}

\begin{proposition}\label{prop:DirectProduct}
Let $S$ and $T$ be reductive LDAGs and $G:=S\times T$. For $V\in\Rep G$, if $_SV={(_SV)}_{s_1,t_1}$ and $_TV={(_TV)}_{s_2,t_2}$, then $V=V_{s_1+s_2,\max\{t_1,t_2\}}$ (see Definition~\ref{def:GV}). 
\end{proposition}
\begin{proof}
We need to show that $V=V_{s_1+s_2}$ and $V=V_{(\max\{t_1,t_2\})}$.
By Proposition~\ref{prop:BasicPropsA}, $V$ embeds into the $G$-module $$U:=\bigoplus_{i=1}^{\dim V}A(i),$$ where $A(i):=A=B\otimes_\K C$, where $B:=\K \{S\}$ and $C:=\K \{T\}$. We will identify $V$ with its image in $U$. Let $\bar{B}_{j}$, $j\in\N$, be subspaces of $B$ such that
$$
B_{j}=B_{j-1}\oplus\bar{B}_{j}.
$$
Similarly, we define subspaces $\bar{C}_{r}\subset C$, $r\in\N$. We have
$$
A=\bigoplus_{j,r} \bar{B}_{j}\otimes_{\K}\bar{C}_{r},
$$
as vector spaces. 
Let $$\pi^i_{jr}:U\to A(i)=A\to \bar{B}_{j}\otimes_{\K}\bar{C}_{r}$$ denote the composition of the projections. Then, the conditions $_SV=(_SV)_{s_1}$ and $_SV=(_SV)_{s_2}$ mean that $$\pi^i_{jr}(V)=\{0\}$$ if $j>s_1$ or $r>s_2$. In particular, $V$ belongs to $$\bigoplus_{i=1}^{\dim V}A(i)_{s_1+s_2}.$$ Hence, $V=V_{s_1+s_2}$. Similarly, using
$$
(B\otimes C)_{(n)}=B_{(n)}\otimes C_{(n)},
$$
one shows $V=V_{(\max\{t_1,t_2\})}$.
\end{proof}

\begin{proposition}\cite[Proof of Lemma~4.5]{diffreductive}\label{prop:RedDecomp}
Let $G$ be a reductive LDAG, $S$ be the differential commutator subgroup of $G^\circ$ (i.e., the Kolchin-closure of the commutator subgroup of $G^\circ$), and $T$ be the identity component of the center of $G^\circ$. 
The LDAG $S$ is semisimple and the multiplication map
$$
\mu: S\times T\to G^\circ,\ (s,t)\mapsto st,
$$
is an epimorphism of LDAGs with a finite kernel.
\end{proposition}

Let $\Rep_{(n)}G$ denote the tensor subcategory of $\Rep G$ generated by $P^n(W)$ (the $n$th total prolongation). The following Proposition shows that $\Rep_{(n)}G$ does not depend on the choice of $W$.

\begin{proposition}\label{prop:RepnGnotdepW}
For all $V\in\Rep G$, $V\in\Rep_{(n)}G$ if and only if $V=V_{(n)}$.
\end{proposition}
\begin{proof}
Suppose $V\in\Rep_{(n)}G$. Since the matrix entries of $P^n(W)$ belong to $A_{(n)}$, we have $V=V_{(n)}$. Conversely, suppose $V=V_{(n)}$. Then $V$ is a representation of the LAG $G_{(n)}$ whose Hopf algebra is $A_{(n)}$. Since $P^n(W)$ is a faithful $A$-comodule, it is a faithful $A_{(n)}$-comodule. Hence, $\Rep G_{(n)}$ is generated by $P^n(W)$.
\end{proof}

If $\tau(G)\Le 0$, then, by~\cite[Section~3.2.1]{MiOvSi}, there exists $n\in\N$ such that 
$$
\Rep G=\left\langle \Rep_{(n)}G\right\rangle_{\otimes}.
$$
The smallest such $n$ will be denoted by $\ord(G)$.  For a $G$-module $V$, let $\Ll(V)$ denote the length of the socle filtration of $V$. In particular, we have $$\Ll(V)\Le\dim V.$$

For a $G$-module $V$, let $\Ll(V)$ denote the length of the socle filtration of $V$. In particular, we have $$\Ll(V)\Le\dim V.$$

\begin{theorem}\label{thm:bound}
Let $G$ be a reductive LDAG with $\tau\big({Z(G)}^\circ\big) \Le 0$ and $T:=Z{(G^\circ)}^\circ$. For all $V\in\Rep G$, we have $V\in\Rep_{(n)}G$, where
\begin{equation}\label{eq:UBound}
n=\max\{\Ll(V)-1, \ord(T)\}.
\end{equation}
\end{theorem}
 
\begin{proof}
Let $V\in\Rep G$. By Proposition~\ref{prop:RepnGnotdepW}, we need to show that $V=V_{(n)}$, where $n$ is given by~\eqref{eq:UBound}. Set $\widetilde{G}:=S\times T$, where $S\subset G$ is the differential commutator subgroup of $G^\circ$. The multiplication map $\mu: \widetilde{G}\to G$ (see Proposition~\ref{prop:RedDecomp}) induces the structure of a $\widetilde{G}$-module on the space $V$, which we will denote by $\widetilde{V}$. By Theorem~\ref{thm:EqualFiltrations},
$$
_S\widetilde{V}={}_S\widetilde{V}_{r}={}_S\widetilde{V}_{(r)},
$$
where $$r=\Ll\big({}_S\widetilde{V}\big)-1=\Ll(_SV)-1.$$ It follows from Proposition~\ref{prop:QuasiIsom} (formula~\eqref{eq:phi*}) and Lemma~\ref{lem:socV} that, if $W\in\Rep G$ is semisimple, then $_SW\in\Rep S$ is semisimple. Hence, $$\Ll(_SV)\Le \Ll(V).$$ Therefore,
$$
_S\widetilde{V}=_S\widetilde{V}_{(s)},\ s:=\Ll(V)-1.
$$
Next, since $\tau(T)\Le 0$, we have $$\Rep T=\Rep_{(t)}T,\quad t:=\ord(T).$$
By Proposition~\ref{prop:RepnGnotdepW}, $_T\widetilde{V}={}_T\widetilde{V}_{(t)}$. Proposition~\ref{prop:DirectProduct} implies 
$$
\widetilde{V}=\widetilde{V}_{(\max\{s,t\})}=\widetilde{V}_{(n)}.
$$
Now, applying Proposition~\ref{prop:QuasiIsom} to $\phi:=\mu$, we obtain $V=V_{(n)}$.
\end{proof}

The following proposition suggests an algorithm to find $\ord(T)$. 

\begin{proposition}\label{prop:FindOrdT}
Let $G\subset\GL(W)$ be a reductive LDAG with $\tau\big({Z(G)}^\circ\big) \Le 0$, where the $G$-module $W$ is semisimple. Set $T:={Z(G^\circ)}^\circ$ and $H:=\overline{G} \subset\GL(W)$. Let
$$
\varrho: H\to\GL(U)
$$ 
be an algebraic representation with $\Ker\varrho=[H^\circ,H^\circ]$.
Then $\ord(T)$ is the minimal number $t$ such that the differential tensor category generated by $_GU\in\Rep G$ coincides with the tensor category generated by $P^t(_GU)\in\Rep G$.
\end{proposition}

\begin{proof}
We have $\varrho(G)=\varrho(T)$ and $\Ker\varrho\cap T$ is finite. Propositions~\ref{prop:QuasiIsom} and~\ref{prop:RepnGnotdepW} complete the proof.
\end{proof}

\section{Computing parameterized differential Galois groups}\label{sec:main}
In this section, we show how the main results of the paper can be applied to constructing  algorithms that compute the maximal reductive quotient of a parameterized differential Galois group and decide if a parameterized Galois is reductive.

\subsection{Linear differential equations with parameters and their Galois theory}\label{sec:PPV}
In this section, we will briefly recall the parameterized differential Galois theory of linear differential equations, also known as the PPV theory \cite{PhyllisMichael}. Let $K$ be a $\Delta' = \{\partial, \partial_1, \ldots , \partial_m\}$-field and 
\begin{eqnarray}\label{EQN:LDE}
\partial Y &= &AY, \ \ A \in \Mn_n(K)
\end{eqnarray}
be a linear differential equation (with respect to $\partial$) over $K$.  A {\em parameterized Picard--Vessiot extension (PPV-extension)} $F$ of $K$ associated with \eqref{EQN:LDE} is a $\Delta'$-field $F \supset K$ such that there exists a $Z \in \GL_n(F)$ satisfying $\partial Z=AZ$, $F^\partial = K^\partial$, and   $F$ is generated over $K$ as a $\Delta'$-field by the entries of $Z$ (i.e., $F = K\langle Z\rangle$).  

The field $K^\partial$ is a $\Delta = \{\partial_1, \ldots, \partial_m\}$-field and, if it is differentially closed, a PPV-extension associated with \eqref{EQN:LDE} always exists and is unique up to a $\Delta'$-$K$-isomorphism~\cite[Proposition~9.6]{PhyllisMichael}. Moreover, if $K^\partial$ is relatively differentially closed in $K$, then $F$ exists as well~\cite[Thm~2.5]{GGO} (although it may not be unique). Some other situations concerning  the existence of $K$ have also been treated in~\cite{Wibmer:existence}. 

If $F = K\langle Z\rangle$ is a PPV-extension of $K$, one defines the {\em parameterized Picard--Vessiot Galois group (PPV-Galois group)} of $F$ over $K$ to be
  $$
 G:= \{\sigma : F\to F\:|\: \sigma\text{ is a field automorphism, } \sigma\delta=\delta\sigma \text{ for all } \delta\in\Delta', \text{ and } \sigma(a)=a,\ a\in K\}.
  $$
For any $\sigma \in G$, one can show that there exists a matrix $[\sigma]_Z \in \GL_n\big(K^\partial\big)$ such that $\sigma(Z) = Z[\sigma]_Z$ and the map $\sigma\mapsto [\sigma]_Z$ is an isomorphism of $G$ onto a differential algebraic subgroup (with respect to $\Delta$) of $\GL_n\big(K^\partial\big)$. 

One can also develop the PPV-theory in the language of modules. 
A finite-dimensional vector space $M$ over the $\Delta'$-field $K$ together with a map $\partial : M \to M$ 
  is called a {\it parameterized differential module} if $$\partial(m_1+m_2) = \partial(m_1)+\partial(m_2)\quad\text{and}\quad \partial(am_1)=\partial(a)m_1+a\partial(m_1),\quad m_1,m_2 \in M,\ a\in K.$$ 
 Let $\{e_1,\ldots,e_n\}$ be a $K$-basis  of $M$ and $a_{ij} \in K$ be such that $\partial(e_i) = -\sum_j a_{ji}e_j$, $1\Le i \Le n$. As in~\cite[Section~1.2]{Michael}, for $v = v_1e_1+\ldots+v_ne_n$, 
 $$\partial(v) = 0\quad \Longleftrightarrow\quad \partial\begin{pmatrix}v_1\\
 \vdots\\
 v_n\end{pmatrix}= A\begin{pmatrix}v_1\\
 \vdots\\
 v_n\end{pmatrix},\quad  A := (a_{ij})_{i,j=1}^n.$$ Therefore, once we have selected a basis, we can associate a linear differential equation of the form \eqref{EQN:LDE} with $M$. Conversely, given such an equation, we define a map $$\partial:K^n\rightarrow K^n,\quad\partial(e_i) = -\sum_j a_{ji}e_j,\quad A=(a_{ij})_{i,j=1}^n.$$  This  makes $K^n$ a parameterized differential module. The collection of parameterized differential modules over $K$ forms an abelian tensor category. In this category, one can define the notion of prolongation $M \mapsto P_i(M)$ similar to the notion of prolongation of a group action as in~\eqref{eq:prolongation}. For example, if $\partial Y = AY$ is the differential equation associated with the module $M$, then (with respect to a suitable basis) the equation associated with $P_i(M)$ is
 $$ \partial Y = \begin{pmatrix} A&\partial_i A\\ 0 & A \end{pmatrix}Y.$$
 Furthermore, if $Z$ is a solution matrix of $\partial Y = AY$, then 
 $$\begin{pmatrix} Z & \partial_i Z \\ 0 & Z\end{pmatrix}$$
 satisfies this latter equation. Similar to the $s^{th}$ total prolongation of a representation, we define the {\em $s^{th}$ total prolongation $P^s(M)$ of a module $M$}  as $$P^s(M) = P_1^sP_2^s\cdot \ldots\cdot P_m^s(M).$$
     If $F$ is a PPV-extension for \eqref{EQN:LDE}, one can define a $K^\partial$-vector space
 $$\omega(M) := \Ker(\partial: M\otimes_KF \to M\otimes_KF).$$
The correspondence $M \mapsto \omega(M)$ induces a functor $\omega$ (called a differential fiber functor) from the category of differential modules  to the category of finite-dimensional vector spaces over $K^\partial$ carrying $P_i$'s into the $P_i$'s (see  \cite[Defs.~4.9,\,4.22]{GGO}, \cite[Definition~2]{OvchTannakian}, \cite[Definition~4.2.7]{Moshe}, \cite[Definition~4.12]{Moshe2012} for more formal definitions). Moreover, 
\begin{eqnarray}\label{eq:tanequiv}
\left(\Rep_G,\mathrm{forget}\right)&\cong&\left(\big\langle P_1^{i_1}\cdot\ldots\cdot P_m^{i_m}(M)\:\big|\: i_1,\ldots,i_m\Ge 0\big\rangle_{\otimes},\omega\right)
\end{eqnarray}
 as differential tensor categories \cite[Thms.~4.27,\,5.1]{GGO}. This equivalence will be further used in the rest of the paper to help explain the  algorithms.  

In Section~\ref{sec:algorithmitself}, we shall restrict ourselves to PPV-extensions of certain special fields. We now describe these fields and give some further  properties of the PPV-theory over these fields.  Let $\K(x)$ be the $\Delta' = \{\partial, \partial_1, \ldots , \partial_m\}$-differential field defined as follows:
\begin{eqnarray}
(\text{i}) & &\mbox{ $\K$ is a differentially closed field with derivations $\Delta = \{\partial_1, \ldots , \partial_m\}$,}\notag\\
(\text{ii}) & & \mbox{ $x$ is transcendental over $\K$, and} \label{eq:Kx}\\
(\text{iii}) & & \mbox{ $\partial_i(x) = 0, \ i = 1, \ldots , m$, $\partial (x) = 1$ and $\partial(a) = 0$ for all $a\in\K$.}\notag
\end{eqnarray}
When one further restricts $\K$, Proposition~\ref{prop:PPVDFGG} characterizes the LDAGs that appear as PPV-Galois groups over such fields. We say that $\K$ is a {\it universal differential field} if, for any differential field  $k_0 \subset \K$ differentially finitely generated over $\Q$ and any differential field $k_1\supset k_0$ differentially finitely generated over $k_0$, there is a differential $k_0$-isomorphism of $k_1$ into $\K$ (\cite[Chapter~III,Section~7]{Kol}). Note that a universal differential field is differentially closed.

\begin{proposition}[cf. \cite{DreyfusDensity,ClaudineMichael}]\label{prop:PPVDFGG} Let $\K$ be a universal $\Delta$-field and let $\K(x)$ satisfy conditions~\eqref{eq:Kx}.  An LDAG $G$ is a parameterized differential Galois group over $\K(x)$ if and only if $G$ is a DFGG. 
\end{proposition}

Assuming that $\K$ is only differentially closed, one still has the following corollary.

\begin{corollary} \label{cor:type0} Let  $\K(x)$ satisfy conditions~\eqref{eq:Kx}. If $G$ is reductive and is a parameterized differential Galois group over $\K(x)$, then $\tau(Z(G^\circ)) \Le 0$.
\end{corollary}
\begin{proof} Let $L$ be a PPV-extension of $\K(x)$ with parameterized differential Galois group $G$ and let $\U$ be a universal differential field containing $\K$ (such a field exists \cite[Chapter~III,Section~7]{Kol}). Since $\K$ is {\it a fortiori} algebraically closed, $ \U\otimes_\K L$ is a domain whose quotient field we denote by $\U L$. One sees that the $\Delta$-constants $\Const$ of $\U L$ are $\U$. We may identify the quotient field $\U(x)$ of  $\U \otimes_\K\K(x)$ with a subfield of $\U L$, and one sees that $\U L$ is a PPV-extension of $\U(x)$.  Furthermore, the parameterized differential Galois group of $\U L$ over $\U(x)$ is $G(\U)$ (see also \cite[Section~8]{GGO}). Proposition~\ref{prop:PPVDFGG} implies that $\G(\U)$ is a DFGG.  Lemma~\ref{lem4} implies that $${\rm tr.~deg.}_{\U} \U\left\langle Z{\left(G^\circ\right)}^\circ \right\rangle < \infty.$$ Since $G^\circ$ is defined over $\K$ and $\K$ is algebraically closed,  ${\rm tr.~deg.}_\K\K\left\langle Z{\left(G^\circ\right)}^\circ \right\rangle < \infty$.  Therefore, $\tau(Z(G^\circ)) \Le 0$.\end{proof}

\subsection{Equivalent statements of reductivity}
In this section, we give a characterization of parameterized differential modules whose PPV-Galois groups are reductive LDAGs, which will be  used in Section~\ref{sec:algorithmitself} to construct the main algorithms.

In this section, let $K$ be a differential field as at the beginning of Section~\ref{sec:PPV}.
Given a parameterized differential module $M$ such that it has a PPV-extension over $K$, let $G$ be its PPV-Galois group. Recall a construction of the ``diagonal part'' of $M$, denoted by $M_{\diag}$, which induces~\cite{OvchTannakian} a differential representation $$\rho_{\diag} : G \to \GL\left(\omega\left(M_{\diag}\right)\right),$$ where $\omega$ is the functor of solutions.
If $M$ is irreducible, we set $M_{\diag} = M$. Otherwise, if $N$ is a maximal differential submodule of $M$, we set $$M_{\diag} = N_{\diag}\oplus M/N.$$ 
Since $M$ is finite-dimensional and $\dim N < \dim M$, $M_{\diag}$ is well-defined above.  Another description of $M_{\diag}$ is: let
\begin{equation}\label{eq:moduleflag}
M = M_0\supset M_1\supset \ldots \supset M_r = \{0\}
\end{equation}
be a complete flag of differential submodules, that is, $M_{i-1}/M_{i}$ are irreducible. We then let $$M_{\diag} = \bigoplus_{i=1}^{r} M_{i-1}/M_{i}.$$
 A version of the Jordan--H\"older Theorem implies that $M_{\diag}$ is unique up to isomorphism.
Note that $M_{\diag}$ is a completely reducible differential module. The complete flag \eqref{eq:moduleflag} corresponds to a differential equation in block upper triangular form
\begin{equation}\label{eq:flageqn} \partial Y = \left(\begin{array}{ccccc} A_r&\ldots &\ldots &\ldots&\ldots \\ 0 & A_{r-1} &\ldots & \ldots&\ldots \\ \vdots&\vdots&\vdots&\vdots&\vdots\\ 0&\ldots&0&A_{2}&\ldots\\0&\ldots &0&0&A_1\end{array}\right)Y,
\end{equation}
where, for each matrix $A_i$, the differential module corresponding to $\partial Y=A_iY$ is irreducible.  The differential module $M_{\diag}$ corresponds to the block diagonal  equation
\begin{equation}\label{eq:diageqn}
 \partial Y = \left(\begin{array}{ccccc} A_r&0 &\ldots &\ldots&0 \\ 0 & A_{r-1} &0 & \ldots&0 \\ \vdots&\vdots&\vdots&\vdots&\vdots\\ 0&\ldots&0&A_{2}&0\\0&\ldots &0&0&A_1\end{array}\right)Y.
 \end{equation}

Furthermore, given a complete flag \eqref{eq:moduleflag}, we can identify the solution space of $M$ in the following way. Let $V$ be the solution space of $M$ and 
\begin{equation}\label{eq:flag}
V=V_0\supset V_1\supset\ldots\supset V_r=\{0\}
\end{equation}
be a complete flag of spaces of $V$ where each $V_i$ is the solution space of $M_i$.  Note that each $V_i$ is a $G$-submodule of $V$ and that all $V_i/V_{i+1}$ are simple $G$-modules.  One then sees that   
$$
V_{\diag}=\bigoplus_{i=1}^rV_{i-1}/V_i.
$$

\begin{proposition}\label{prop:equivdiag} Let $$\mu : G \to \overline{G}\big/\Ru{\left(\overline{G}\right)}\to\overline{G} \subset \GL(\omega(M))$$  be the morphisms (of LDAGs) corresponding to a Levi decomposition of $\overline{G}$. Then $\rho_{\diag} \cong \mu$.
\end{proposition}
\begin{proof}
	Since $\rho_{\diag}$ is completely reducible, $\omega\big(M_{\diag}\big)$ is a completely reducible $\rho_{\diag}{\left(\overline{G}\right)}$-module.  Therefore, $\rho_{\diag}{\left(\overline{G}\right)}$ is a reductive LAG \cite[Chapter~2]{SpringerInv}. Hence, $$\Ru{\left(\overline{G}\right)} \subset \Ker\rho_{\diag},$$ where $\rho_{\diag}$ is considered as a map from $\overline{G}$. On the other hand, by definition, $\Ker\rho_{\diag}$ consists of unipotent elements only.
Therefore, since  $\Ker\rho_{\diag}$ is a normal subgroup of $\overline{G}_M$ and connected by \cite[Corollary~8.5]{Waterhouse},
\begin{equation}\label{eq:kerdiag}
\Ker \rho_{\diag} = \Ru{\left(\overline{G}\right)}.
\end{equation} Since all Levi  $K^\partial$-subgroups of $\overline{G}$ are conjugate (by $K^\partial$-points of $\Ru{\left(\overline{G}_M\right)}$) \cite[Theorem~VIII.4.3]{Hochschild},~\eqref{eq:kerdiag} implies that $\rho_{\diag}$
is equivalent to $\mu$.
\end{proof}

\begin{corollary}\label{cor:inj} In the notation of Proposition~\ref{prop:equivdiag}, $\rho_{\diag}$ is faithful if and only if 
\begin{equation}\label{eq:modbyux}
G \to \overline{G}/\Ru{\left(\overline{G}\right)}
\end{equation} is injective.
\end{corollary}
\begin{proof}
Since $\rho_{\diag}\cong\mu$ by Proposition~\ref{prop:equivdiag}, faithfulness of $\rho_{\diag}$ is equivalent to that of $\mu$, which is precisely the injectivity of~\eqref{eq:modbyux}.
\end{proof}

\begin{proposition}\label{prop:equiv} The following statements are equivalent:
\begin{enumerate}
\item\label{it:311} $\rho_{\diag}$ is faithful, 
\item\label{it:312} $G$ is a reductive LDAG,
\item\label{it:313} there exists $q\Ge 0$ such that 
\begin{equation}\label{eq:pX} 
M \in {\left\langle P^q\left(M_{\diag}\right)\right\rangle}_{\otimes}.
\end{equation}
\end{enumerate}
\end{proposition}
\begin{proof}
\eqref{it:311} implies \eqref{it:313} by \cite[Proposition~2]{OvchRecoverGroup} and \cite[Corollary~3 and~4]{OvchTannakian}. If a differential representation $\mu$ of the LDAG $G$ is not faithful, so are the objects in the category ${\left\langle P^q(\mu)\right\rangle}_\otimes$ for all $q \Ge 0$. Using the equivalence of neutral differential Tannakian categories from~\cite[Theorem~2]{OvchTannakian}, this shows that~\eqref{it:313} implies~\eqref{it:311}.

 If $\rho_{\diag}$ is faithful, then $G$ is reductive by the first part of the proof of \cite[Theorem~4.7]{diffreductive}, showing that~\eqref{it:311} implies~\eqref{it:312}. Suppose now that $G$ is a reductive LDAG.  
 Since $\Ru{\left(\overline{G}\right)}\cap G$ is a connected normal unipotent differential algebraic subgroup of $G$, it is equal to $\{\id\}$. Thus,~\eqref{eq:modbyux} is injective and, by Corollary~\ref{cor:inj},~\eqref{it:312} implies~\eqref{it:311}.
\end{proof}

\subsection{Algorithm}\label{sec:algorithmitself}
In this section, we will assume that $\K(x)$ satisfies conditions \eqref{eq:Kx} and, furthermore, that $\K$ is computable, that is, one can effectively carry out the field operations and effectively apply the derivations.  We will describe an algorithm for calculating the maximal reductive quotient $G/\Ru(G)$ of the PPV- Galois group $G$ of any $\partial Y= AY, \ A \in \GL_n(\K(x))$ and an algorithm to decide if $G$ is reductive, that is, if $G$ equals this maximal reductive quotient. 

\subsubsection{Ancillary Algorithms.}  We begin by describing algorithms to solve the following problems which arise in our two main algorithms. 

\begin{algorithm}\label{alg:A}{\it Let $K$  be a computable algebraically closed field and $H \subset \GL_n(K)$ be a reductive LAG  defined over $K$. Given the defining equations for $H$, find defining equations for $H^\circ$ and $Z(H^\circ)$ as well as defining equations for normal simple algebraic groups $H_1, \ldots, H_\ell$ of $H^\circ$ such that the homomorphism $$\pi:H_1 \times \ldots \times H_\ell  \times Z(H^\circ) \rightarrow H^\circ$$ is surjective with a finite kernel.}  \cite{EisHunVas} gives algorithms for finding Gr\"obner bases of the radical of a polynomial ideal and of the prime ideals appearing in a minimal decomposition of this ideal.  Therefore, one can find the defining equations of $H^\circ$. Elimination properties of Gr\"obner bases allow one to compute $$Z(H^\circ) =\big\{h\in H^\circ \ | \ ghg^{-1} = h \mbox{ for all } g \in H^\circ\big\}.$$ We may write $H^\circ = S\cdot Z(H^\circ)$ where $S=[H^\circ,H^\circ]$  is semisimple. A theorem of Ree~\cite{Ree} states that every element of a connected semisimple algebraic group is a commutator, so $$S = \big\{[h_1,h_2] \:|\: h_1,h_2 \in H^\circ\big\}.$$ Using the elimination property of Gr\"obner bases, we see that one can compute defining equations for $S$. We know that $S = H_1\cdot\ldots\cdot H_\ell$ for some simple algebraic groups $H_i$.  We now will  find the $H_i$. 
Given the defining ideal $J$ of $S$,  the Lie algebra $\is$ of $S$  is $$\big\{s \in \Mn_n(K)\ | \ f(I_n +  \epsilon s) = 0 \mod \epsilon^2 \mbox{ for all } f \in J\big\},$$ where $\epsilon$ is a new variable. This $K$-linear space is also computable via Gr\"obner bases techniques. In \cite[Section~1.15]{deGraaf}, one finds algorithms to decide if $\is$ is simple and, if not, how to decompose $\is$ into a direct sum of simple ideals $\is = \is_1 \oplus \ldots \oplus \is_\ell$.  Note that each $\is_i$ is the tangent space of a normal simple algebraic subgroup $H_i$ of $S$ and $S = H_1\cdot\ldots\cdot H_\ell$. Furthermore, $H_1$ is the identity component of $$\{ h \in S \ | \ {\rm Ad}(h)(\is_2\oplus \ldots \oplus \is_\ell) = 0\},$$ and this can be computed via Gr\"obner bases methods.  Let $S_1$ be the identity component of $$\{ h \in S \ | \ {\rm Ad}(h)(\is_1) = 0\}.$$ We have $S = H_1\cdot S_1$, and we can proceed by induction to determine $H_2, \ldots, H_\ell$ such that $S_1 = H_2\cdot\ldots\cdot H_\ell$.  The groups $Z(H^\circ)$ and $H_1, \ldots, H_\ell$ are what we desire.
\end{algorithm}
 \begin{algorithm}\label{alg:B} {\it Given $A \in \Mn_n(\K(x))$, find defining equations for the PV-Galois
group $H \subset \GL_n(\K)$ of the differential equation $\partial Y = AY$. When $H$ is finite, construct the PV-extension associated with this equation.} A general algorithm to compute PV-Galois groups is given by Hrushovski \cite{Hrushovski}.  When $H$ is assumed to be reductive, an algorithm is given in~\cite{CoSi97b}. An algorithm to find all algebraic solutions of a differential equation is classical (due to Painlev\'e and Boulanger) and is described in \cite{singer_alg, singer_liouvsoln}.
\end{algorithm}
\begin{algorithm}\label{alg:C}
  {\it Given  $A \in \Mn_n(\K(x))$ and the fact that the PPV-Galois group $G$  of the differential equation $\partial Y = AY$ satisfies $\tau(G) \Le 0$, find the defining equations of $G$. } An algorithm to compute this is given in \cite[Algorithm~1]{MiOvSi}.
  \end{algorithm}
\begin{algorithm}\label{alg:D} {\it Assume that we are given an algebraic extension $F$ of $\K(x)$, a matrix $A \in \Mn_n(F)$, the defining equations for the PV-Galois group $G$ of the equation $\partial Y = AY$ over $F$ and the defining equations for a normal algebraic  subgroup $H$ of $G$.  Find an integer $\ell$, a faithful representation $\rho:G/H \rightarrow \GL_\ell(\K)$ and a  matrix $B \in \Mn_\ell(F)$ such that the equation $\partial Y = BY$ has PV-Galois group $\rho(G/H)$. }  

The usual proof (\cite[Section~11.5]{Humphreys}) that there 
exists an $\ell$ and a faithful rational representation  $\rho: G/H \rightarrow \GL_\ell(\K)$ is constructive; that is, if  $V \simeq \K^n$ is a faithful $G$-module and we are given the defining equations for $G$ and $H$, then, using direct sums, subquotients, duals, and tensor products, one can construct a $G$-module $W\simeq \K^\ell$ such that the map $\rho:G\rightarrow \GL_\ell(\K)$ has kernel $H$. 

 Let $M$ be the differential module associated with $\partial Y = AY$. Applying the same constructions to $M$ yields a differential module $N$.  The Tannakian correspondence implies that the action of $G$ on the associated solution space is (conjugate to) $\rho(G)$.
\end{algorithm}

\begin{algorithm}\label{alg:E} {\it Assume that we are given $F$, an algebraic extension of $\K(x)$, and  $A \in  \Mn_n(F)$, and $B_1, \ldots , B_\ell \in F^n$. Let 
\begin{eqnarray*}\label{eq:param}
W & = & \big\{(Z, c_1, \ldots , c_\ell ) \ | \ Z \in F^n, c_1, \ldots , c_\ell \in \K\ \ \text{\rm and}\ \ \partial Z + AZ = c_1B_1 + \ldots + c_\ell B_\ell \big\}.
\end{eqnarray*}
Find a $\K$-basis of $W$.}   Let $F[\partial]$ be the ring of differential operators with coefficients in $F$. Let $$C = I_n\partial + A\in \Mn_n(F[\partial]).$$ We may write $\partial Z + AZ = c_1B_1 + \ldots + c_\ell B_\ell $ as $$CZ = c_1B_1 + \ldots + c_\ell B_\ell.$$ Since $ F[\partial]$ has a left and right division algorithm (\cite[Section~2.1]{Michael}), one can row and column reduce the matrix $C$, that is, find  a left invertible matrix $U$ and a right invertible matrix $V$ such that $UCV = D$ is a diagonal matrix.  We then have that $(Z, c_1, \ldots ,c_\ell) \in W$ if and only if $X = (V^{-1}Z, c_1, \ldots , c_\ell)$ satisfies $$DX = c_1UB_1 + \ldots + c_\ell UB_\ell.$$  Since $D$ is diagonal, this is equivalent to finding bases of scalar parameterized equations $$Ly = c_1b_1+\ldots +c_\ell b_\ell, \quad L\in F[\partial],\ b_i \in K.$$  \cite[Proposition~3.1 and Lemma~3.2]{singer_liou} give a method to solve this latter problem.   We note that, if $A \in \K(x)$ and $\ell = 1$, an algorithm for finding solutions with entries in $\K(x)$ directly without having to diagonalize  is given in~\cite{barkatou}.
\end{algorithm}
\begin{algorithm}\label{alg:F} {\it Let $A \in \Mn_n(\K(x))$ and let $M$ be the differential module associated with $\partial Y = AY$. Find a basis of $M$ so that the associated differential equation $\partial Y = BY, \ B \in \Mn_n(\K(x))$, is as in~\eqref{eq:flageqn}, that is, in block upper triangular form with the blocks on the diagonal corresponding to irreducible modules.}  We are asking to ``factor'' the system $\partial Y = AY$.  Using cyclic vectors, one can reduce this problem to factoring linear operators of order $n$, for which there are many algorithms (cf. \cite[Section~4.2]{Michael}). A direct method is also given in~\cite{grigoriev90b}.
\end{algorithm}

\begin{algorithm}\label{alg:G} {\it Suppose that we are given $F$, an algebraic extension of $\K(x)$,  $A \in  \Mn_n(F)$, and the defining equations of the  PV-Galois group $H$ of $\partial Y = AY$.  Assuming that $H$ is  a simple LAG, find the PPV-Galois $G$ group of  $\partial Y = AY$.  }   Let $\calD$ be the $\K$-span of $\Delta$. A Lie $\K$-subspace $\calE$ of  $\calD$ is a $\K$-subspace such that,  if $D,D' \in \calE$, then $$[D,D'] = DD'-D'D \in \calE.$$ We know that the group $G$ is a Zariski-dense subgroup of $H$. The Corollary to \cite[Theorem~17]{CassidyClassification} states that there is a Lie $\K$-subspace $\calE \subset \calD$ such that $G$ is conjugate to  $H\big(\K^{\calE}\big)$. Therefore, to describe $G$, it suffices to find $\calE$.  Let
$$
W = \big\{(Z, c_1, \ldots , c_m ) \ | \ Z \in \Mn_n(F) =F^{n^2},\, c_1, \ldots , c_m \in \K\ \mbox{ and }\ \partial Z + [Z,A] = c_1\partial_1 A + \ldots + c_m \partial_m A \big\}.
$$
The algorithm described in~\ref{alg:E} allows us to calculate $W$.  We claim that we can take
\begin{eqnarray}\label{cale}
\calE = \big\{ c_1\partial_1 + \ldots + c_m \partial_m \ | \mbox{ there exists } Z\in \GL_n(F) \mbox{ such that } (Z, c_1, \ldots , c_m ) \in W\big\}.
\end{eqnarray}
Note that this $\calE$ is a Lie $\K$-subspace of $\calD$.  To see this, it suffices to show that, if $D_1, D_2 \in \calE$, then $[D_1,D_2] \in \calE$.  If
 $$ \partial B_1 + [B_1,A] = D_1A\quad\text{and}\quad  \partial B_2 + [B_2,A] = D_2A  \ \ \ \mbox{ for  some }\ B_1,B_2 \in \GL_n(F),$$
 then a calculation shows that
 $$ \partial B + [B,A] = [D_1,D_2]A,\quad\text{where}\quad B = D_1B_2 - D_2B_1 - [B_1,B_2].$$ 
In particular, \cite[Section~0.5, Propostions~6 and~7]{KolDAG} imply that $\calE$ has a $\K$-basis of {\it commuting} derivations $\big\{\dbar_1, \ldots , \dbar_t\big\}$ that extends to a basis of commuting derivations $\big\{\dbar_1, \ldots , \dbar_m\big\}$ of $\calD$.

To show that $G$ is conjugate to $H\big(\K^{\calE}\big)$ we shall need  the following concepts and results.  Let $\deltabar' = \big\{\dbar, \dbar_1, \ldots, \dbar_m\big\}$ and $k$ be a $\deltabar'$-field. Let $\deltabar = \big\{ \dbar_1, \ldots, \dbar_m\big\}$ and $\Sigma \subset \deltabar$. Assume that $ C= k^\dbar$ is differentially closed.
\begin{definition} Let $A\in \Mn(k)$. We say $\dbar Y = AY$ is {\it integrable with respect to $\Sigma$} if, for all $\dbar_i \in \Sigma$, there exists $A_i \in \Mn_n(k)$ such that 
\begin{eqnarray}
\dbar A_j - \dbar_j A& = &[A,A_j]\ \ \text{for all}\ \ \dbar_j\in \Sigma\  \mbox{ and,  }\label{eq:com1}\\
 \dbar_i A_j - \dbar_jA_i &=& [A_i,A_j]\ \ \text{for all}\ \ \dbar_i, \dbar_j \in \Sigma \label{eq:com2}
\end{eqnarray}
\end{definition} 

The following characterizes integrability in terms of the behavior of the PPV-Galois group.
\begin{proposition} \label{prop:int} Let $K$ be the PPV-extension of $k$ for $\dbar Y = AY$ and let $G \subset \GL_n(C)$ be the PPV-Galois group. The group $G$ is conjugate to a subgroup of $\GL_n\big(C^\Sigma\big)$ if and only if $\dbar Y = AY$ is integrable with respect to $\Sigma$.\end{proposition}
\begin{proof} Assume that $G$ is conjugate to a subgroup of $\GL_n\big(C^\Sigma\big)$ and let $B \in \GL_n(C)$ satisfy $$BGB^{-1} \subset \GL_n\big(C^\Sigma\big).$$ Let $Z \in \GL_n(K)$ satisfy $\dbar Z = AZ$ and  $W = ZB^{-1}$.  For any $V \in \GL_n(K)$ such that $\dbar V = AV$ and $\sigma \in G$, we will denote by $[\sigma]_V$ the matrix in $\GL_n(C)$ such that $\sigma(V) = V{[\sigma]}_V$.  We have 
$$  \sigma(W) = Z{[\sigma]}_Z B^{-1} = ZB^{-1}B {[\sigma]}_Z B^{-1} = W{[\sigma]}_W,$$
so $${[\sigma]}_W  = B{[\sigma]}_ZB^{-1} \in \GL_n\big(C^\Sigma\big).$$  A calculation shows that $A_i := \dbar_iW\cdot W^{-1}$ is left fixed by all $\sigma \in G$ and so lies in $\Mn_n(k)$.  Since the $\dbar_i$ commute with $\dbar$ and each other, we have that the $A_i$ satisfy \eqref{eq:com1} and  \eqref{eq:com2}.
 
Now assume that  $\dbar Y = AY$ is integrable with respect to $\Sigma$ and, for convenience of notation, let $\Sigma = \big\{\dbar_1, \ldots , \dbar_t\big\}$. We first note that since $C$ is differentially closed with respect to $\Delta$, the field $C^\Sigma$ is differentially closed with respect to $\Pi =   \big\{\dbar_{t+1}, \ldots , \dbar_m\big\}$ (in fact, $C^\Sigma$ is also differentially closed with respect to $\Delta$, see~\cite{Omar}). Note that $C^\Sigma = k^{\{\dbar\}\cup\Sigma}$. Let $$R = k\{Z, 1/(\det Z)\}_{\deltabar'}$$ be the PPV-extension ring of $k$  for the integrable system
\begin{eqnarray}\label{eq:intsys}
\dbar Y & = & AY  \label{eq:intsys1}\\
\dbar_i  Y & = & A_iY , i= 1, \ldots t.\label{eq:intsys2}
\end{eqnarray}
The ring $R$ is a $\deltabar'$-simple ring generated both as a $\Pi$-differential ring and as a $\deltabar$-differential ring by the entries of $Z$ and $1/\det Z$.  Therefore, $R$ is also the PPV-ring for the single equation \eqref{eq:intsys1}, (\cite[Definition 6.10]{CharlotteMichael}). 

Let $L$ be the quotient field of $R$. The group $G$ of $\deltabar'$-automorphisms of $L$ over $k$ is both the PPV-group of  the system \eqref{eq:intsys1} \eqref{eq:intsys2} and of the single equation \eqref{eq:intsys1}. In the first case, we see that the matrix representation of this group with respect to $Z$ lies in $\GL_n\big(C^\Sigma\big)$ and therefore the same is true in the second case. Since $C^\Sigma$ is differentially closed,  the PPV-extension $K = k\langle U\rangle$ is $k$-isomorphic to $L$ as $\deltabar'$-fields.  This isomorphism will take $U$ to $ZD$ for some $D \in \GL_n(C)$ and so the matrix representation of the  PPV-group of $K$ over $k$ will be conjugate to a subgroup of $\GL_n\big(C^\Sigma\big)$.

One can also argue as follows. First note that $C$ is also $\Sigma$-differentialy closed by \cite{Omar}.
For every $\deltabar$-LDAG $G' \subset \GL_n(C)$ with defining ideal $$I \subset C\{X_{ij},1/\det\}_\deltabar,$$ let  $G'_\Sigma$ denote the $\Sigma$-LDAG with defining ideal $$J:=I\cap C\{X_{ij},1/\det\}_\Sigma.$$ Then $G'$ is conjugate to $\Sigma$-constants if and only if $G'_\Sigma$ is. Indeed, the former is equivalent to the existence of $D \in \GL_n(C)$ such that, for all $i$,$j$, $1\Le i,j\Le n$ and $\partial\in\Sigma$, we have $\partial\big(DX_{ij}D^{-1}\big)_{ij} \in I$, which holds if and only if $\partial\big(DXD^{-1}\big)_{ij} \in J$.

Let $K = k\langle Z\rangle_{\deltabar'}$.  
The $\Sigma$-field $K_{\Sigma} := k\langle Z\rangle_{\{\dbar\}\cup\Sigma}$ is a $\Sigma$-PPV extension for $\dbar Y = AY$ by definition. As in \cite[Proposition~3.6]{PhyllisMichael}, one sees that $G_\Sigma$ is its $\Sigma$-PPV Galois group. Finally, $G_\Sigma$ is conjugate to $\Sigma$-constants if and only if $\dbar Y = AY$ is integrable with respect to $\Sigma$ by \cite[Proposition~3.9]{PhyllisMichael}.
\end{proof}

\begin{corollary}\label{cor:533} Let $K$ be the PPV-extension of $k$ for $\dbar Y = AY$ and $G \subset \GL_n(C)$ be the PPV-Galois group. Then $G$ is conjugate to a subgroup of $\GL_n\big(C^\Sigma\big)$ if and only if, for every $\dbar_i \in \Sigma$, there exists $A_i \in  \Mn_n(k)$ such that 
$\dbar A_j +[A_j,A] =  \dbar_j A $.\end{corollary}
\begin{proof} In \cite[Theorem~4.4]{GO}, the authors show that $G$ is conjugate to a subgroup of $\GL_n\big(C^\Sigma\big)$ if and only if for each $\dbar_i \in \Sigma$, $G$ is conjugate to a subgroup of $\GL_n\big(C^{\dbar_i}\big)$.  Two applications of Proposition~\ref{prop:int} yields the conclusion.\end{proof}

Applying Corollary~\ref{cor:533} to $\dbar=\partial$ and the commuting basis $\Sigma = \big\{\dbar_1, \ldots , \dbar_t\big\}$ of $\calE$, implies that $G$ is conjugate to $H\big(\K^\calE\big)$. 
\end{algorithm}

Sections~\ref{sec:alg1} and~\ref{sec:alg2} now present the two algorithms described in the introduction.

\subsubsection{An algorithm to compute the  maximal reductive quotient $G/\Ru(G)$ of a PPV-Galois group $G$.} \label{sec:alg1}Assume that we are given a matrix $A \in \Mn_n(\K)$. Let $H$ be the PV-Galois group of this equation. We proceed as follows taking into account the following general principle. For every normal algebraic subgroup $H'$ of $H$ and $B \in \Mn_\ell(\K)$, if $H/H'$ is the PV-Galois group of $\partial Y=BY$, then $G/(G\cap H')$ is its PPV-Galois group, which follows from~\ref{alg:D}.

\begin{step}{\it Reduce to the case where $H$ is reductive.} Using~\ref{alg:F},  we find an equivalent differential equation as in \eqref{eq:flageqn} whose matrix is in block upper triangular form where the modules corresponding to the diagonal blocks are irreducible.  We now consider the block diagonal Equation \eqref{eq:diageqn}.  This latter equation has PPV-Galois group $G/\Ru(G)$.
\end{step}
\begin{step} {\it Reduce to the case where $G$ is connected and semisimple.}  We will show that it is sufficient to be able to compute the PPV-Galois group of an equation $\partial Y = \overline{A} Y$assuming $\overline{A}$ has entries in an algebraic extension of $\K(x)$, assuming  we have the defining equations of the PV-Galois group of $\partial Y = \overline{A} Y$ and  assuming this PV-Galois group is connected and semisimple. 

Using~\ref{alg:B}, we compute the defining equations of the  PV-Galois group $H$ of $\partial Y = AY$ over $\K(x)$.  Using~\ref{alg:A}, we calculate the defining equations for $H^\circ$ and $Z\big(H^\circ\big)$ as well as defining equations for normal simple algebraic groups $H_1, \ldots, H_\ell$ of $H^\circ$ as in~\ref{alg:A}.  Note that $$H^\circ = S_H\cdot Z\big(H^\circ\big),$$ where $S_H = H_1\cdot\ldots\cdot H_\ell$ is the commutator subgroup of $H^\circ$.  Note that $$S_G = \big[G^\circ,G^\circ\big]$$ is Zariski-dense in $S_H$.  Using~\ref{alg:D}, we construct a differential equation $\partial Y = BY$ whose PV-Galois group is $H/H^\circ$.  This latter group is finite, so this equation has only algebraic solutions, and, again using~\ref{alg:B}, we can construct a finite extension $F$ of $\K(x)$ that is the PV-extension corresponding to $\partial Y = BY$. The PV-Galois group of $\partial Y = AY$ over $F$ is $H^\circ$. 

Since we have the defining equations of $Z(H^\circ)$,~\ref{alg:D} allows us to construct a representation $$\rho:H^\circ \rightarrow H^\circ/Z\big(H^\circ\big)$$ and a   differential equation $\partial Y = \overline{B}Y$, $\overline{B}$ having entries in $F$, whose     PV-Galois group is $\rho(H^\circ)$. Note that $\rho\big(G^\circ\big)$ is the PPV-Galois group of $\partial Y = \overline{B}Y$ and is Kolchin-dense in $\rho\big(H^\circ\big)$.  Therefore, $\rho\big(G^\circ\big)$ is connected and semisimple. Let us assume that we can find defining equations of  $\rho(G^\circ)$. We can therefore compute defining equations of $\rho^{-1}\big(\rho\big(G^\circ\big)\big)$.  The group $$\rho^{-1}\big(\rho\big(G^\circ\big)\big) \cap S_H$$ normalizes $\big[G^\circ,G^\circ\big]$ in $S_H$. By Lemma~\ref{lem:normal}, we have $$\rho^{-1}\big(\rho\big(G^\circ\big)\big) \cap S_H = S_G .$$  Therefore, we can compute the defining equations of $S_G$.  

To compute the defining equations of $G$, we proceed as follows.   Using~\ref{alg:D}, we compute a differential equation $\partial Y = \widetilde{B}Y$, $\widetilde{B}$ having entries in $\K(x)$, whose PV-group is $H/S_H$. The PPV-Galois group of this equation is $L = G/S_G$.  By Lemma~\ref{lem4}, this group has differential type at most $0$, so~\ref{alg:C} implies that we can find the defining equations of $L$.  Let $$\widetilde{\rho}:H \rightarrow H/S_H.$$ We claim that $$G =  \widetilde{\rho}^{-1}(L) \cap N_H\big(S_G\big).$$  Clearly, $$G \subset \widetilde{\rho}^{-1}(L) \cap N_H\big(S_G\big).$$ Now let $$h \in \widetilde{\rho}^{-1}(L) \cap N_H\big(S_G\big).$$  We can write $h = h_0g$ where $g \in G$ and $h_0 \in S_H$. Furthermore, $h_0$ normalizes $S_G$.  Lemma~\ref{lem:normal} implies that $h_0 \in S_G$ and so $h\in G$.  Since we can compute the defining equations of $S_G$, we can compute the defining equations of $N_H(S_G)$. Since we can compute $\widetilde{\rho}$ and the defining equations of $L$, we can compute the defining equations of  $\widetilde{\rho}^{-1}(L)$, and so we get the defining equations of $G$.
All that remains is to prove the following lemma.
\begin{lemma}\label{lem:normal} Let $G$ be a Zariski-dense differential subgroup of a semisimple linear algebraic group $H$. Then
\begin{enumerate}
\item $Z(H) \subset G$, and
\item $N_H(G) = G.$
\end{enumerate}
\end{lemma}
\begin{proof} \cite[Theorem~15]{CassidyClassification} implies that $$H = H_1\cdot\ldots\cdot H_\ell\quad \text{and}\quad G = G_1\cdot\ldots\cdot G_\ell,$$ where each $H_i$ is a normal simple algebraic subgroup of $H$ with $[H_i,H_j] = 1$ for $i\neq j$ and each $G_i$ is Zariski-dense in $H_i$ and normal in $G$.  Therefore, it is enough to prove the claims when $H$ itself is a simple algebraic group.  In this case, let us assume that $H \subset \GL(V)$, where $H$ acts irreducibly on $V$.  Schur's Lemma implies that the center of $H$ consists of scalar matrices and, since $H=(H,H)$, these matrices have determinant $1$.  Therefore, the matrices are of the form $\zeta I$ where $\zeta $ is a root of unity.    \cite[Theorem~19]{CassidyClassification} states that there  is a Lie $K$-subspace $\calE$ of $\calD$, the $\K$-span of $\Delta$, such that $G$ is conjugate to $H\big(\K^\calE\big)$. Since the roots of unity are constant for any derivation, we have that the center of $H$ lies in $G$.

To prove $N_H(G) = G$, assume $G = H\big(\K^\calE\big)$ and let $g \in G$ and $h\in N_H(G)$.  For any $\dbar \in \calE$, we have
$$ 0 = \dbar\big(h^{-1}gh\big) = -h^{-1}\dbar (h) h^{-1}  g h + h^{-1} g \dbar (h).$$
Therefore, $\dbar (h) h^{-1}$ commutes with the elements of $G$ and so must commute with the elements of $H$. Again by Schur's Lemma, $\dbar (h) h^{-1}$
is a scalar matrix.  On the other hand, $\dbar (h) h^{-1}$ lies in the Lie algebra of $H$ (\cite[Section~V.22, Proposition~28]{Kol}) and so the trace of $\dbar (h) h^{-1}$ is zero.  Therefore, $\dbar (h) h^{-1} = 0$.  Since $\dbar (h) =0$ for all $\dbar \in \calE$, we have $h \in G$.\end{proof}
\end{step}
\begin{step} {\it Computing $G$ when $G$ is connected and semisimple.} We have reduced the problem to calculating the PPV-Galois group $G$ of an equation $\partial Y = AY$ where the entries of $A$ lie in an algebraic extension $F$ of $\K(x)$ and where we know the equations of the PV-Galois $H$ group of this equation over $F$.  Let $$H= H_1\cdot\ldots\cdot H_\ell\quad \text{and}\quad G = G_1\cdot\ldots\cdot G_\ell,$$ where the $H_i$ are simple normal subgroups of $H$ and $G_i$ is Zariski-dense in $H_i$.  Using~\ref{alg:D}, we construct, for each $i$,  an equation $\partial Y = B_iY$ with $B_i \in \Mn_n(F)$  whose PV-Galois group is $H/\bar{H}_i$, where $$\bar{H}_i  = H_1\cdot\ldots\cdot H_{i-1}\cdot H_{i+1}\cdot\ldots\cdot H_\ell$$ and a surjective homomorphism $\pi_i: H \rightarrow H/\bar{H}_i$.  Note that $H/\bar{H}_i$ is a connected simple LAG.  Therefore,~\ref{alg:G} allows us to calculate the PPV-Galois group $\bar{G}_i$ of $\partial Y = B_iY$.  We claim that $$G_i = \pi_i^{-1}\big(\bar{G}_i\big) \cap H_i.$$ To see this, note that $\bar{H}_i \cap H_i$ lies in the center of $H_i$ and, therefore, must lie in $G_i$ by Lemma~\ref{lem:normal}.  Therefore, we have defining equations for each $G_i$ and so can construct defining equations for $G$.
\end{step}

\subsubsection{An algorithm to decide if the PPV-Galois group of a parameterized linear differential equation is reductive.}\label{sec:alg2} Let $\K(x)$ be as in~\eqref{eq:Kx}.  Assume that we are given a differential equation $\partial Y = AY$ with $A \in \Mn_n(\K(x))$.  Using the solution to~\ref{alg:F} above, we may assume that $A$ is in block upper triangular form as in~\eqref{eq:flageqn} with the blocks on the diagonal corresponding to irreducible differential modules. Let $A_{\diag}$ be the corresponding diagonal matrix as in \eqref{eq:diageqn},  let $M, G$ and $M_{\diag},  G_{\diag}$ be the differential modules and PPV-Galois groups associated with $\partial Y = AY $ and $\partial Y = A_{\diag}Y$, respectively.  Of course, $$G_{\diag} \simeq G/\Ru(G),$$ so $G$ is reductive if and only if $G_{\diag} \simeq G$.  

This implies via the Tannakian equivalence that  the differential tensor category generated by $M_{\diag}$ is a subcategory of the differential tensor category generated by $M$ and that $G$ is reductive if and only if these categories are the same.    The differential tensor category generated by a module $M$  is the usual tensor category generated by all the total  prolongations $P^s(M)$ of that module.  

From this, we see that $G$ is a reductive LDAG if and only if $M$ belongs to the tensor category  generated by some total prolongation $P^s(M_{\diag})$.  Therefore, to decide if $G$ is reductive, it suffices to find algorithms to solve problems~\ref{alg:H} and~\ref{alg:I} below.

\begin{algorithm}\label{alg:H}{\it Given differential modules $M$ and $N$, decide if $M$ belongs to the tensor category generated by $N$.} Since we are considering the tensor category and not the {\em differential} tensor category, this is a question concerning nonparameterized differential equations.  Let $K_N,K_M,K_{M\oplus N}$ be PV-extensions associated with the corresponding differential modules and let $G_M, G_N, G_{N\oplus M}$ be the corresponding PV-Galois groups.  The following four conditions are easily seen to be equivalent:
\begin{enumerate}
\item[(a)] $N$ belongs to the tensor category generated by $M$;
\item[(b)] $K_N \subset K_M$ considered as subfields of $K_{M\oplus N}$;
\item[(c)] $K_{M\oplus N} = K_M$;
\item[(d)] the canonical projection $\pi: G_{M\oplus N} \subset G_M \oplus G_N \rightarrow G_M$ is injective (it is always surjective).
\end{enumerate}
Therefore, to solve~\ref{alg:H}, we apply the algorithmic solution of~\ref{alg:B} to calculate $G_{M\oplus N}$ and $G_M$ and, using Gr\"obner bases, decide if  $\pi$ is injective.
\end{algorithm}

\begin{algorithm}\label{alg:I}{\it Given $M$ and $M_{\diag}$ as above, calculate an integer $s$ such that, if $M$ belongs to the differential tensor  category generated by $M_{\diag}$, then $M$ belongs to the tensor category generated by $P^s\big(M_{\diag}\big)$.} We will apply Theorem~\ref{thm:bound} and  Proposition~\ref{prop:FindOrdT}.  Note that, since the PPV-Galois group $G_{\diag}$ associated to $M_{\diag}$ is reductive, Lemma~\ref{lem4} implies that we may apply these results to   $G_{\diag}$. Theorem~\ref{thm:bound}
implies that such a bound is given by the integer 
$$\max\{\Ll(V)-1, \ord(T)\} $$
where $V$ is a solution space associated with $M_{\diag}$ and $T = {Z\big(G_{\diag}^\circ\big)}^\circ$.  As noted in the discussion  preceding Theorem~\ref{thm:bound}, $$\Ll(V) \Le \dim_\K(V) = \dim_{\K(x)} M_{\diag}.$$ Proposition~\ref{prop:FindOrdT} implies that $\ord(T)$ can be bounded in the following way. Using the algorithm to solve~\ref{alg:B}, we calculate the defining equations of the PV-Galois  group $H_{\diag}$ associated with $M_{\diag}$  and then calculate the defining equations of $H_{\diag}^\circ$ and  $\big[H_{\diag}^\circ,H_{\diag}^\circ\big]$ (as in~\ref{alg:A}). Using the solution to~\ref{alg:D}, one calculates a differential equation $\partial Y = BY$ whose PV-Galois group is $$H\big/\big[H_{\diag}^\circ,H_{\diag}^\circ\big].$$ Denote the associated differential module by $N$. Proposition~\ref{prop:FindOrdT} implies that $\ord(T)$ is the smallest value of $t$ so that the differential tensor category generated by $N$ coincides with the tensor category generated by $P^t(N)$. The following conditions are easily seen to be equivalent
\begin{enumerate}
\item[(a)] The differential tensor category generated by $N$ coincides with the tensor category generated by $P^t(N)$.
\item[(b)] The tensor category generated by $P^t(N)$ coincides with the tensor  category generated by $P^{t+1}(N)$.
\item[(c)] $P^{t+1}(N)$ belongs to the tensor category generated by $P^t(N)$.
\end{enumerate}
Therefore, to bound $\ord(T)$, one uses the algorithm of~\ref{alg:H} to check  for $t=0, 1, 2, \ldots$ if $P^{t+1}(N)$ belongs to the tensor category generated by $P^t(N)$ until this event happens (see also \cite[Section~3.2.1,  Algorithm~1]{MiOvSi}). As noted  in the discussion preceding  Theorem~\ref{thm:bound}, this procedure eventually halts. Taking the maximum of this $t$ and $\dim_{\K(x)}M - 1$ yields the desired $s$.
\end{algorithm}

\section{Examples}\label{sec:examples}
\addtocounter{subsection}{1}
In this section, we will illustrate both Theorem~\ref{thm:EqualFiltrations} and our main algorithm. In Example~\ref{ex:sharp}, we will show that the bound in Theorem~\ref{thm:EqualFiltrations} is sharp. Example~\ref{ex:alg} is an illustration of the algorithm.
\begin{example} Following \cite[Ex.~4.18]{MinOvRepSL2}, let $$V =\Span_\K\left\{1, x_{11}'x_{21}-x_{11}x_{21}',x_{11}'x_{22}-x_{21}'x_{12},x_{12}'x_{22}-x_{12}x_{22}',x_{11}'x_{22}-x_{12}'x_{21}\right\} \subset A,$$
where 
\begin{equation}\label{eq:ASL2}
A := \K\{x_{11},x_{12},x_{21},x_{22}\}\big/[x_{11}x_{22}-x_{12}x_{21}-1],
\end{equation}
which induces
the following differential representation of $\SL_2$:
$$
\SL_2(\U) \ni \begin{pmatrix}
a& b\\
c& d
\end{pmatrix} \mapsto
\begin{pmatrix}
1 & a'c-ac'& a'd-bc'& b'd-bd'&a'd'-b'c'\\
0&a^2 & ab & b^2& ab'-a'b \\
0&2ac&ad+bc& 2bd& 2(ad'-bc')\\
0&c^2& cd& d^2&cd'-c'd\\
0&0&0&0&1
\end{pmatrix}
$$
under the right action of $\SL_2$ on $A$.
Since the length of the socle filtration for $V$ is $3$, let $n=2$. Theorem~\ref{thm:EqualFiltrations} claims that $V\in\left\langle P^2\left(V_{\diag}\right)\right\rangle_\otimes$.
We will show that, in fact, 
\begin{equation}\label{eq:VinPdiagV}
V \in \left\langle P\left(V_{\diag}\right)\right\rangle_\otimes.
\end{equation}
Indeed, by the Clebsch--Gordon formula for tensor products of irreducible representations of $\SL_2$, the usual irreducible representation $U=\Span_\K\{u,v\}$ of $\SL_2$  is a direct summand of $V_{\diag}\otimes V_{\diag}$. Moreover, $$V\subset (P(U)\oplus P(U))\otimes(P(U)\oplus P(U))$$ under the embedding  $$U\oplus U\to A,\quad (au+bv,cu+dv) \mapsto ax_{11}+bx_{12}+cx_{21}+dx_{22},$$
which implies~\eqref{eq:VinPdiagV}.
\end{example}
\begin{example}\label{ex:sharp}
Consider the first prolongations $P(V)$
 of the usual (irreducible) representation $r: \SL_2 \to \GL(V)$ of dimension $2$:
$$
P(r) : \SL_2 \ni A \mapsto\begin{pmatrix}
A&A'\\
0&A
\end{pmatrix}.
$$
The length of the socle filtration is $2$, and we tautologically have $$P(V) \in \left\langle P^{2-1}\left(P(V)_{\diag}\right)\right\rangle_{\otimes}.$$ 
Note that $$P(V) \notin {\left\langle P(V)_{\diag}\right\rangle}_\otimes$$
as every object of ${\left\langle  P(V)_{\diag}\right\rangle}_\otimes={\langle V\rangle}_\otimes$ is completely reducible \cite[Thm~4.7]{diffreductive} but $P(V)$ is not completely reducible \cite[Proposition~3]{OvchRecoverGroup}, \cite[Theorem~4.6]{GO}. By Proposition~\ref{prop:EasyInlusion}, for all $n\Ge 0$, 
\begin{equation}\label{eq:Pnsoc}
{P^n(V)}_n \subset \soc^{n+1}P^n(V).
\end{equation}
Since $r^\vee =: \rho : V\to V\otimes_\K A_0$, where $A$ is defined in~\eqref{eq:ASL2}, for all $n\Ge 0$, $$P^n(\rho) : P^n(V)\to P^n(V)\otimes_\K A_n$$
(see~\eqref{eq:defAn}). Therefore, ${P^n(V)}_n = P^n(V)$. Since $P^n(V) \supset \soc^{n+1}P^n(V)$,~\eqref{eq:Pnsoc} implies that
$$
P^n(V) = \soc^{n+1}P^n(V).
$$
Therefore, the length of the socle filtration of $P^n(V)$ does not exceed $n+1$.
If 
\begin{equation}\label{eq:inclusion}
P^{n+1}(V) \in \left\langle P^n(V)\right\rangle_\otimes,
\end{equation} then, for all $q > n$, $P^q(V) \in \left\langle P^n(V)\right\rangle_\otimes$,
which implies that 
\begin{equation}\label{eq:equal}
\left\langle P^i(V)\:\big|\: i \Ge 0\right\rangle_\otimes = \left\langle P^n(V)\right\rangle_\otimes.\end{equation}
By~\cite[Proposition~2.20]{Deligne},~\eqref{eq:equal} implies that $A$ is a finitely generated $\K$-algebra, which is not the case. Therefore,~\eqref{eq:inclusion} does not hold. Thus, the bound in Theorem~\ref{thm:EqualFiltrations} is sharp.
\end{example}

We will now illustrate how the algorithm works.  Let $\Const$ denote the differential closure of $\Q$ with respect to a single derivation $\partial_t$.  In the following examples, we consider the differential equations over the field $\K(x) = \Const(x)$ with derivations $\Delta' = \{ \partial_x,\partial_t\}$ and $\Delta = \{ \partial_t\}$. 
\begin{example}\label{ex:alg} As in \cite[Ex.~3.4]{MiOvSi},
consider the equation $\partial_xY = AY$ where
$$A = \begin{pmatrix}
1& \frac{t}{x}+\frac{1}{x+1}\\
0&1
\end{pmatrix},
$$
whose PV-group is 
\begin{equation}\label{eq:PV1}
\left\{\begin{pmatrix}a & b\\
0& a
\end{pmatrix}\:\Big|\: a,\,b \in \U, a\neq 0\right\} \simeq  \Gm \times \Ga,
\end{equation}
which is not reductive. Let $M$ be the corresponding differential module.
Using our algorithm, we will test whether the PPV-Galois group $G$ of $\partial_xY=AY$ is reductive.
We have 
$$
A_{\diag} = \begin{pmatrix}
1&0\\
0&1
\end{pmatrix},
$$
and the PV and PPV-Galois groups of $\partial_xY=A_{\diag}Y$ are $\Gm$ and $\Gm(\Const)$, respectively; see \cite[Proposition~3.9(2)]{PhyllisMichael}. Therefore, $$\ord \big(G/\Ru(G)\big) =\ord\big(\Gm(\Const)\big) = 1.$$ The matrix of $M\oplus P^1\big(M_{\diag}\big)$ with respect to the appropriate basis is
$$
\begin{pmatrix}
1& \frac{t}{x}+\frac{1}{x+1}&0&0&0&0\\
0&1&0&0&0&0\\
0&0&1&0&0&0\\
0&0&0&1&0&0\\
0&0&0&0&1&0\\
0&0&0&0&0&1
\end{pmatrix},
$$
which is not completely reducible by~\eqref{eq:PV1}. Therefore, its PV group is not isomorphic to $\Gm$, the PV group of $M_{\diag}$. Thus, $G$ is not reductive. In fact, $G$ is calculated in~\cite[Ex.~3.4]{MiOvSi} yielding
$$G= \left\{\left.\begin{pmatrix} e & f\\0&e \end{pmatrix} \in \Gm(\Const)\times \Ga(\Const)\ \right| \ \partial_te= 0, \  \partial_t^2f = 0 \right\}.\qedhere
$$
\end{example}

\begin{example}  Consider the equation 
\begin{equation}\label{eq:sl2gm}
\partial_x^2(y) + 2xt\partial_x(y)+ty=0.
\end{equation}
The PPV-Galois group of this equation lies in $\GL_2$. One can make a standard substitution (\cite[Exc.~1.35.5]{Michael}) resulting in a new equation having PPV-Galois group in $\SL_2$. Once we know the PPV-Galois group of this new equation, results of \cite{Carlos} allow us to construct the PPV-group of the original equation. In our example, the appropriate substitution is   $y= z e^{-\int xt}$.
 We find that $z$ satisfies the equation
\begin{equation}\label{eq:sl2}
\partial_x^2(y) - \big(1/4(2xt)^2+(2xt)'/2-t\big)y = 0\quad\Longleftrightarrow\quad \partial_x^2(y)-(xt)^2y=0,
\end{equation} which now has PPV-Galois group in $\SL_2$, and  $e^{-\int xt}$ satisfies the equation
\begin{equation}\label{eq:gm}
\partial_x(y)+((2xt)/2)y=0 \quad\Longleftrightarrow\quad \partial_x(y)+(xt)y=0,
\end{equation} which has PPV-Galois group in $\GL_1 = \Gm$. We shall refer to Equations~\eqref{eq:sl2} and~\eqref{eq:gm} as the auxiliary equations.
A calculation on {\sc Maple} using the {\tt kovacicsols} procedure of the {\tt DEtools} package shows that the PV Galois group $H$ of~\eqref{eq:sl2} is $\SL_2$. Since, for all $0\ne n\in\ZZ$, $\U(x)$ has no solutions of $$\partial_x(y)+(nxt)y=0,$$ the PV Galois group of~\eqref{eq:gm} is $\Gm$. Therefore, by \cite[Section~3.4]{Carlos}, the PV Galois group of~\eqref{eq:sl2gm} is $$\GL_2 \cong (\SL_2\times\Gm)\big/\{1,-1\}.$$  Hence, the PPV-Galois group $G$ of~\eqref{eq:sl2gm} is of the form
$$
G = (G_1\times G_2)\big/\{1,-1\} \subset H,
$$
where
$G_2$ is Zariski-dense in $\Gm$ and $G_1$ is conjugate in $\GL_2$ either to $\SL_2$ or $\SL_2(\Const)$. We will now calculate $G_1$ and $G_2$. For the former, note that the matrix form of~\eqref{eq:sl2} is
$$
\partial_xY=\begin{pmatrix}
0&1\\
(xt)^2&0
\end{pmatrix}Y.
$$
Since, for the matrix $$
B := \begin{pmatrix}
0&\frac{x}{2t}\\
\frac{tx^3}{2}&\frac{1}{2t}
\end{pmatrix},
$$
which can be found using the {\tt dsolve} procedure of {\sc Maple},
one has $\partial_x(B)-\partial_t(A)=[A,B]$,~\eqref{eq:sl2} is completely integrable and, therefore, $G_1$ is conjugate to $\SL_2(\Const)$. To find $G_2$, compute the first prolongation of~\eqref{eq:gm}:
$$
A_1 := \begin{pmatrix}
-xt&-x\\
0&-xt
\end{pmatrix}.
$$
Setting
$$
C:= \begin{pmatrix}
1&1\\
\frac{-2}{x^2}&\frac{-2}{x^2+t}
\end{pmatrix},
$$
we see that
$$
C^{-1}A_1C - C^{-1}\partial_x(C)=\begin{pmatrix}
\frac{2-x^2 t}{x}&0\\
0&\frac{x(2-x^2 t-t^2)}{x^2+t}
\end{pmatrix}.
$$
Hence, the differential equation corresponding to $A_1$ is completely reducible. Therefore,  $G_2 = \Gm(\Cat)$, that is,
$$
G \cong \GL_2(\Const).
$$ Note that $C$ can be found using the {\tt dsolve} procedure of {\sc Maple}.
\end{example}
\begin{example} Starting with 
\begin{equation}\label{eq:almostgl2}
\partial_x^2(y) - \frac{2t}{x}\partial_x(y)=0,
\end{equation}
the auxiliary equations will be
$$
\partial_x^2(y)-\frac{t(t+1)}{x^2}y=0\quad\text{and}\quad\partial_x(y)=\frac{t}{x}y.
$$
The PPV-Galois group of the latter equation is 
$$
G_2 := \left\{g\in\Gm \ \big| \  (\partial_t^2g)g - (\partial_t g)^2 = 0\right\}.
$$
For the former equation, a calculation using {\tt dsolve} from {\sc Maple} shows that there is no $B \in \Mn_2(\U(x))$ such that $\partial_x(B)-\partial_t(A)=[A,B]$, where
$$
A := \begin{pmatrix}
0&1\\
\frac{t(t+1)}{x^2}&0
\end{pmatrix},
$$
which implies that this equation is not completely integrable. Therefore, $G_1 = \SL_2$. Thus, the PPV-Galois group of~\eqref{eq:almostgl2} is
$$
\left\{g \in \GL_2\:\big|\: \left(\partial_t^2\det(g)\right)\det(g) - (\partial_t \det(g))^2 = 0\right\}.\qedhere
$$ 
\end{example}

\section*{Funding}
A.M. was supported 
by the ISF grant 756/12.
A.O. was supported by the NSF grant CCF-0952591. M.F.S. was supported by the NSF grant CCF-1017217

\setlength{\bibsep}{0.0pt}
\bibliographystyle{abbrvnat}
\small
\bibliography{compreductive}

\end{document}